\documentclass[11pt, letterpaper, reqno]{amsart}

\usepackage{times}
\usepackage[T1]{fontenc}
\usepackage{mathrsfs, enumerate}
\usepackage{latexsym}
\usepackage{tikz}
\usepackage{amsmath,amsfonts,amsthm,amssymb,amscd,mathtools}
\usepackage{chngcntr}
\counterwithin{figure}{section} 
\newtheorem{thm}{Theorem}[section]
\newtheorem*{thm*}{Theorem}
\newtheorem{innercustomthm}{Theorem}
\newenvironment{customthm}[1]
  {\renewcommand\theinnercustomthm{#1}\innercustomthm}
  {\endinnercustomthm}

\newtheorem{lem}[thm]{Lemma}
\newtheorem{prop}[thm]{Proposition}

\theoremstyle{definition}
\newtheorem{defi}[thm]{Definition}

\theoremstyle{remark}
\newtheorem{rmk}[thm]{Remark}

\numberwithin{equation}{section}

\newcommand{\RR}{\ensuremath{\mathbb{R}}}

\newcommand{\ZZ}{\ensuremath{\mathbb{Z}}}

\newcommand{\NN}{\mathbb{N}}

\newcommand{\TT}{\mathbb{T}}

\newcommand{\sB}{\mathcal{B}}

\newcommand{\cC}{\mathcal{C}}
\newcommand{\cCbar}{\widetilde{\mathcal{C}}}

\newcommand{\sS}{{\mathcal{S}}}

\newcommand{\floor}[1]{\lfloor{#1}\rfloor}

\newcommand{\ceil}[1]{\lceil{#1}\rceil}

\newcommand{\um}{{\mu}}
\newcommand{\vm}{{\nu}}

\newcommand{\cD}{\mathcal{D}}

\newcommand{\alphap}{{\alpha'}}
\newcommand{\betap}{{\beta'}}

\newcommand{\talpha}{{\sigma}}
\newcommand{\trho}{{\tau}}

\DeclareMathOperator{\im}{im} 

\definecolor {UMblue}  {RGB}{0, 39, 76}
\definecolor {UMmaize} {RGB}{255, 203, 5}

\definecolor {color_b}{RGB}{255,0,0}
\definecolor {color_c}{RGB}{20, 200, 30}
\definecolor {color_a}{RGB}{0,0,255}

\title[Dilated floor functions having nonnegative commutator I.]
{Dilated floor functions having nonnegative commutator\\
I. Positive and mixed sign dilations}

\author{J.\ C.\ Lagarias}
\address{Dept.\ of Mathematics, University of Michigan,
Ann Arbor, MI 48109--1043}
\email{lagarias@umich.edu}
\author{D.\ H.\ Richman}
\address{Dept.\ of Mathematics, University of Michigan,
Ann Arbor, MI 48109--1043}
\email{hrichman@umich.edu}


\thanks{Work of the  first author was partially supported by NSF grants DMS-1401224 and DMS-1701576,
and work of the second author was partially supported by NSF grant DMS-1600223.}

\begin{document}
\begin{abstract}
In this paper and its sequel we classify the set $S$ of all real parameter pairs 
$(\alpha,\beta)$ such that the dilated floor functions 
$f_\alpha(x) = \lfloor{\alpha x}\rfloor$ and $f_\beta(x) = \lfloor{\beta x}\rfloor$ 
have a nonnegative commutator, i.e. 
$ [ f_{\alpha}, f_{\beta}](x) = \lfloor{\alpha \lfloor{\beta x}\rfloor}\rfloor - 
\lfloor{\beta \lfloor{\alpha x}\rfloor}\rfloor \geq 0$ 
for all real $x$. 
The relation $[f_\alpha,f_\beta]\geq 0$ induces a preorder on the set of non-zero dilation factors 
$\alpha, \beta$, which extends the divisibility partial order on positive integers. 
This paper treats the cases where at least one of the 
dilation parameters $\alpha$ or $\beta$ is nonnegative. 
The analysis of the positive dilations case is related to the theory of Beatty sequences 
and to the Diophantine Frobenius problem in two generators.
\end{abstract} 
\maketitle

\tableofcontents
\bibliographystyle{amsplain}

\section{Introduction}

The {\em floor function}  
$\floor x $
rounds a real number down to the nearest integer. It is a basic operation of discretization.
Given a real parameter $\alpha$, we define the {\em dilated floor function} $f_\alpha(x)= \floor{\alpha x}$; it 
performs discretization at the length scale $\alpha^{-1}$.
Dilated floor functions have recently played a role in describing Ehrhart quasi-polynomials for
dilates of rational simplices, treating them as
step-polynomials, cf. \cite[Example 13]{BBDKV:2012}, \cite[Definition 23]{BBKV:2013}, \cite{BBDKV:2016}.

It is a fundamental question to understand the interaction of discretization at two different scales.
This problem arises in computer graphics, for example, when one rescales figures already discretized.
Mathematically it  leads to study of  compositions of two such functions 
$ f_{\alpha}\circ f_{\beta}(x) = \floor{\alpha \floor{\beta x}}$. 
A number of identities relating dilated floor functions with different dilation factors are given in 
Graham, Knuth and Patashnik \cite[Chap. 3]{GKP94}.
They raised research problems  concerning compositions of dilated floor functions, 
cf.  \cite[Research problem 50, p.101]{GKP94},
some later addressed in  Graham and O'Bryant \cite{GO10}. 
Compositional iterates of a single dilated  floor function were studied in 1994 by Fraenkel \cite{Fra94}, and iterates of floor functions and fractional part functions in 1995 by H\o{a}land and Knuth \cite{HK95}.

Dilated floor functions generally  do not  commute under composition of functions; 
the order of taking   successive discretizations matters. 
In other words, the compositional commutator 
$$
 [ f_{\alpha}, f_{\beta}](x) := f_{\alpha}\circ f_{\beta} (x) - f_{\beta}\circ f_{\alpha}(x) 
 = \floor{\alpha \floor{\beta x}} -  \floor{\beta \floor{\alpha  x}}
$$ 
is generally not the zero function.
The commutator $[f_{\alpha}, f_{\beta}]$
is always a bounded generalized polynomial in the sense of Bergelson and Leibman \cite{BerL07}, 
who studied ergodic properties of generalized polynomials under iteration.

Recently the   authors  together  with T. Murayama \cite{LMR16}   
obtained  necessary and sufficient conditions 
for two dilated floor  functions to commute 
under composition, i.e. to have
 $[f_{\alpha}, f_{\beta}] =0$. 

\setcounter{thm}{-1}
\begin{thm}[Commuting Dilated Floor Functions]
\label{thm:commuting}
 The complete set of all $(\alpha, \beta) \in \RR^2$ such that
  \[
    \lfloor \alpha \lfloor \beta x\rfloor \rfloor = \lfloor \beta \lfloor \alpha x\rfloor \rfloor \quad \text{for all }x\in\RR
  \]
  consists of:
  \begin{enumerate}
    \item[(i)]
      three continuous families $(\alpha, \alpha)$, $(\alpha, 0)$, $(0, \beta)$ for
      all $\alpha, \in \RR$ resp. $\beta \in \RR$.
          \item[(ii)]
      the infinite discrete family 
      \[
        \left\{ (\alpha, \beta) = \left(\frac{1}{m}, \frac{1}{n}\right) : m, n
        \ge 1 \right\},
      \]
      where $m,n$ are positive integers. (The families overlap when $m=n$.)
  \end{enumerate}
\end{thm}
\noindent We let $S_0$ denote this set of dilation factors $(\alpha,\beta)$ satisfying $[f_\alpha,f_\beta]=0$.
The  interesting feature of this result is the existence of the  
discrete family  of ``exceptional'' solutions given in (ii),
which appear to have number-theoretic significance.

This paper addresses the more general problem of determining the set $S$ of all values
 $(\alpha,\beta) \in \RR^2$ that satisfy the one-sided inequality 
\begin{equation}\label{eqn:ineq} 
\floor{\alpha \floor{\beta x}} \geq \floor{\beta \floor{\alpha x}} \quad\text{for all }x\in\RR.
\end{equation}
In what follows we will  sometimes
abbreviate \eqref{eqn:ineq} by writing $[f_\alpha, f_\beta] \geq 0$, 
with the understanding that this inequality should hold identically for all $x\in\RR$.
When $\alpha = 0$ or $\beta = 0$ it is easy to see the inequality  is satisfied, 
as both sides are equal to zero.
If $\alpha $ and $\beta $ differ in sign, the situation is also easy to analyze and
the answer does not depend on the magnitudes $|\alpha|, |\beta|$; 
the inequality holds exactly when $\alpha<0$ and $\beta>0$
(see Section \ref{sec:basic} for details).
The cases when $\alpha$ and $\beta$ have the same sign, however, lead to an intricate and interesting answer.


%
%

\subsection{Classification theorems}

The nonnegative commutator set $S$ contains the coordinate axes, where $\alpha=0$ or $\beta=0$,
by inspection, since  $\floor{\alpha \floor{\beta x}}=\floor{\beta \floor{\alpha x}} = 0.$
 The description of the nonnegative commutator set $S$ for nonzero $\alpha, \beta$ 
is given by the following three classification theorems, according to the signs of
$\alpha$ and $\beta$.
The first case is for mixed sign dilations, which  is  straightforward.

\begin{thm}[Mixed Sign Dilations Classification]
\label{thm:mixed-sign}
Suppose dilation factors $\alpha$ and $\beta$ have opposite signs.
\begin{enumerate}
\item[(a)]
 If $\alpha < 0$ and $\beta > 0$, then the commutator relation $[f_{\alpha}, f_\beta]\geq 0$ 
is  satisfied.
\item[(b)]
 If $\alpha > 0$ and $\beta < 0$, then the commutator relation $[f_\alpha, f_{\beta}]\geq 0$ 
is not satisfied.
\end{enumerate}
\end{thm}

The main part of the analysis is the  classification theorem for same sign dilations,
which splits in two cases: positive dilations and negative dilations. 
Positive dilations are classified in the following result.

\begin{thm}[Positive Dilations Classification]
\label{thm:positive}
Given dilation factors $\alpha, \beta > 0$, the inequality
\[ 
\floor{\alpha \floor{\beta x}} \geq \floor{\beta \floor{\alpha x}} 
\quad\text{for all }x\in\RR
\]
holds if and only if there are integers $m, n\geq 0$, not both $0$,  such that
\begin{equation}
\label{eq:positive}
m  \alpha \beta + n\alpha = \beta.
\end{equation}
\end{thm}

We describe the positive dilation part of the  set $S$ geometrically 
in more detail in Section \ref{subsec:21}, 
and we prove Theorem \ref{thm:positive} in Sections \ref{sec:positive-S}
and \ref{sec:positive-N}.

The negative dilation case  has a more complicated classification. We state it here, 
and will prove it  in Part II \cite{LagR:2018b}.
 
\begin{thm}[Negative Dilations Classification]
\label{thm:negative}
Given dilation factors $\alpha, \beta < 0$, the inequality
\[ \floor{\alpha \floor{\beta x}} \geq \floor{\beta \floor{\alpha x}} \quad\text{for all }x\in\RR \]
holds if and only if one (or more) of the following conditions holds:\\

\begin{enumerate}[(i)]
\item\label{it:quad3-a} There are integers $m \geq 0$, $n\geq 1$ such that
\begin{equation}\label{eq:quad3-1}
 m \alpha \beta - n \beta  = -\alpha.
\end{equation}

\item\label{it:quad3-b} There are coprime integers $p,q\geq 1$ such that
\begin{equation}\label{eq:quad3-2}
\alpha = - \frac{q}{p}, \quad  - \frac{1}{p} \leq \beta \leq 0.
\end{equation}

\item\label{it:quad3-c} There are coprime integers $p, q \geq 1$ and integers $m\geq 0, n\geq 1, r\geq 2$ such that
\begin{equation}\label{eq:quad3-3}
\alpha = -\frac{q}{p}, \quad \beta = -\frac{1}{p}\left(1 + \frac{1}{r}\left(  \frac{m}{p} + \frac{n}{q}-1\right) \right)^{-1}
\end{equation}
with $0< \frac{m}{p} + \frac{n}{q} <1$.
\end{enumerate}
\end{thm}

These  classification theorems describe 
the solution set $S$ 
as a countable union of  real semi-algebraic sets  having dimensions $2$, $1$, or $0$. 
The mixed dilation case  consists of  a dimension $2$ component
($\alpha<0$ and $\beta>0$).  
The positive  dilation cases and negative dilation cases of $S$ both  contain
 $1$-dimensional families of solutions, which 
are each parts of real algebraic curves, consisting either
of straight half-lines or  arcs of rectangular hyperbolas. 
In the positive dilation case these families cover $S$. 
In the negative dilation case there is an additional family  
of $1$-dimensional curves (case {\em(ii)} of Theorem \ref{thm:negative}), 
which consist of finite length vertical line segments 
(of varying length) at all rational values of $\alpha$.
Finally, the  negative dilation case has  in addition   
a countable number of $0$-dimensional components (case {\em(iii)} of Theorem \ref{thm:negative}).
These have rational coordinates, so we term  these points ``sporadic rational solutions.''

Figure \ref{fig:solutions} gives a schematic 
plot of these solutions. 

\begin{figure}[h]
\begin{tikzpicture}[scale=0.8]
  \draw[->] (-8.6,0) -- (6.4,0) node[right] {$\alpha$};
  \draw[->] (0,-6.4) -- (0,6.4) node[above] {$\beta$};
  \draw[scale=4] (1,0.05) -- (1,-0.05) node[below] {1};
  \draw[scale=4] (0.05,1) -- (-0.05,1) node[left] {1};
  
  \fill[scale=4,gray,nearly transparent] (0,0) -- (0,1.5) -- (-2.1,1.5) -- (-2.1,0) -- cycle;
  
  \draw[xscale=4,yscale=4,domain=-1.6:1.6,smooth,variable=\x, color_a] plot ({\x},{\x});
  \foreach \m in {1/2,1/3,1/4,1/5,1/6,1/7,1/8}{
    \draw[xscale=4,yscale=4,domain=0:1.6,smooth,variable=\x, color_a] plot ({\x*\m},{\x}); }
    
  \foreach \m in {1/2,1/3,1/4,1/5,1/6,1/7,1/8}{
  \draw[xscale=4,yscale=4,domain=-2.1:0,smooth,variable=\x, color_a] plot ({\x},{\x*\m}); }
  
  \foreach \x in {1,1/2,1/3,1/4,1/5,1/6,1/7,1/8}{
  \draw[xscale=4,yscale=4,domain=0:1.6,smooth,variable=\y, color_b]  plot ({\x},{\y}); }
  
  \foreach \x in {1,2}{
  \draw[xscale=4,yscale=4,domain=-1:0,smooth,variable=\y, color_b]  plot ({-\x},{\y}); }
  \foreach \x in {1/2,3/2}{
  \draw[xscale=4,yscale=4,domain=-1/2:0,smooth,variable=\y, color_b]  plot ({-\x},{\y}); }
  \foreach \x in {1/3,2/3,4/3,5/3}{
  \draw[xscale=4,yscale=4,domain=-1/3:0,smooth,variable=\y, color_b]  plot ({-\x},{\y}); }
  \foreach \x in {1/4,3/4,5/4,7/4}{
  \draw[xscale=4,yscale=4,domain=-1/4:0,smooth,variable=\y, color_b]  plot ({-\x},{\y}); }
  \foreach \x in {1/5,2/5,3/5,4/5,6/5,7/5,8/5,9/5}{
  \draw[xscale=4,yscale=4,domain=-1/5:0,smooth,variable=\y, color_b]  plot ({-\x},{\y}); }
  \foreach \x in {1/6,5/6,7/6,11/6}{
  \draw[xscale=4,yscale=4,domain=-1/6:0,smooth,variable=\y, color_b]  plot ({-\x},{\y}); }
  \foreach \x in {1/7,2/7,3/7,4/7,5/7,6/7,8/7,9/7,10/7,11/7,12/7,13/7}{
  \draw[xscale=4,yscale=4,domain=-1/7:0,smooth,variable=\y, color_b]  plot ({-\x},{\y}); }
  
  \foreach \x in {23/11}{
  \draw[xscale=4,yscale=4,domain=-1/11:0,smooth,variable=\y, color_b]  plot ({-\x},{\y});}
  
  \foreach \s in {4}{  
    \foreach \y in {4/3,6/5,8/7,10/9,12/11,14/13,16/15,18/17}{ 
      \node[circle,inner sep=0.6pt,fill,color_b] at (-2*\s,-\y*\s) {}; }
      
    \foreach \y in {     9/14,  
                         9/16, 
                    6/11, 9/17, 12/23, 15/29 
      }{ \node[circle,inner sep=0.5pt,fill,color_b] at (-3/2*\s,-\y*\s) {};  }
      
    \foreach \y in { 4/9,      8/21 
                    ,4/11,6/17,8/23,10/29 
      }{ \node[circle,inner sep=0.5pt,fill,color_b] at (-2/3*\s,-\y*\s) {};  }
      
    \foreach \y in { 8/15,      16/39   
                    ,           16/42   
                    ,8/19,12/31,16/43   
                    ,8/21                
                    ,     12/34          
                    ,8/23                
      }{ \node[circle,inner sep=0.5pt,fill,color_b] at (-4/3*\s,-\y*\s) {};  }
      
    \foreach \y in {                        25/63, 35/93 
                     , 10/21,        20/51, 25/66      
                     , 10/23, 15/38, 20/53, 25/68, 30/83  
                     ,        15/41     
                     , 10/29, 15/44    
      }{ \node[circle,inner sep=0.5pt,fill,color_b] at (-5/3*\s,-\y*\s) {};  }
      
    \foreach \y in {      9/28   
                    ,6/19,9/31   
                    ,     9/32   
                    ,6/23,9/35   
      }{ \node[circle,inner sep=0.5pt,fill,red] at (-3/4*\s,-\y*\s) {};  }
      
    \foreach \y in {      15/44,   
                    10/29,15/49,   
                          15/52,  
                    10/33,15/53,  
                    10/39,15/59   
      }{ \node[circle,inner sep=0.5pt,fill,red] at (-5/4*\s,-\y*\s) {};  }
      
    \foreach \jr in {24/5   
                    ,20/3,20/6   
                    ,17/2,17/3   
                    ,3/2           
                    ,    2/3, 2/5   
                    ,1/2       
      }{ \node[circle,inner sep=0.5pt,fill,red] at (-7/4*\s,{-\s*1/(4-\jr/7)}) {};  }
    \foreach \y in {4/15,6/25,9/35   
                   ,4/17   
                   ,4/19   
      }{ \node[circle,inner sep=0.5pt,fill,red] at (-2/5*\s,-\y*\s) {}; }
  }
  
  \draw[xscale=4,yscale=4,domain=0:1.61,smooth,variable=\y, color_c] plot ({\y/(\y+1)},{\y}); 
  \draw[xscale=4,yscale=4,domain=0:1.61,smooth,variable=\y, color_c] plot ({\y/(1+2*\y)},{\y});
  \draw[xscale=4,yscale=4,domain=0:1.6,smooth,variable=\y, color_c] plot ({\y/(1+3*\y)},{\y});
  \draw[xscale=4,yscale=4,domain=0:1.6,smooth,variable=\y, color_c] plot ({\y/(2+\y)},{\y}); 
  
  \draw[xscale=4,yscale=4,domain=-2.1:0,smooth,variable=\x, color_c] plot ({\x},{\x/(1-\x)}); 
  \draw[xscale=4,yscale=4,domain=-2.1:0,smooth,variable=\x, color_c] plot ({\x},{\x/(1-2*\x)});
  \draw[xscale=4,yscale=4,domain=-2.1:0,smooth,variable=\x, color_c] plot ({\x},{\x/(1-3*\x)});
  \draw[xscale=4,yscale=4,domain=-2.1:0,smooth,variable=\x, color_c] plot ({\x},{\x/(2-\x)}); 
  
\end{tikzpicture}
\caption{All solutions $S$ 
to the nonnegative commutator relation $[f_\alpha,f_\beta]\geq 0$.}
\label{fig:solutions}
\end{figure}
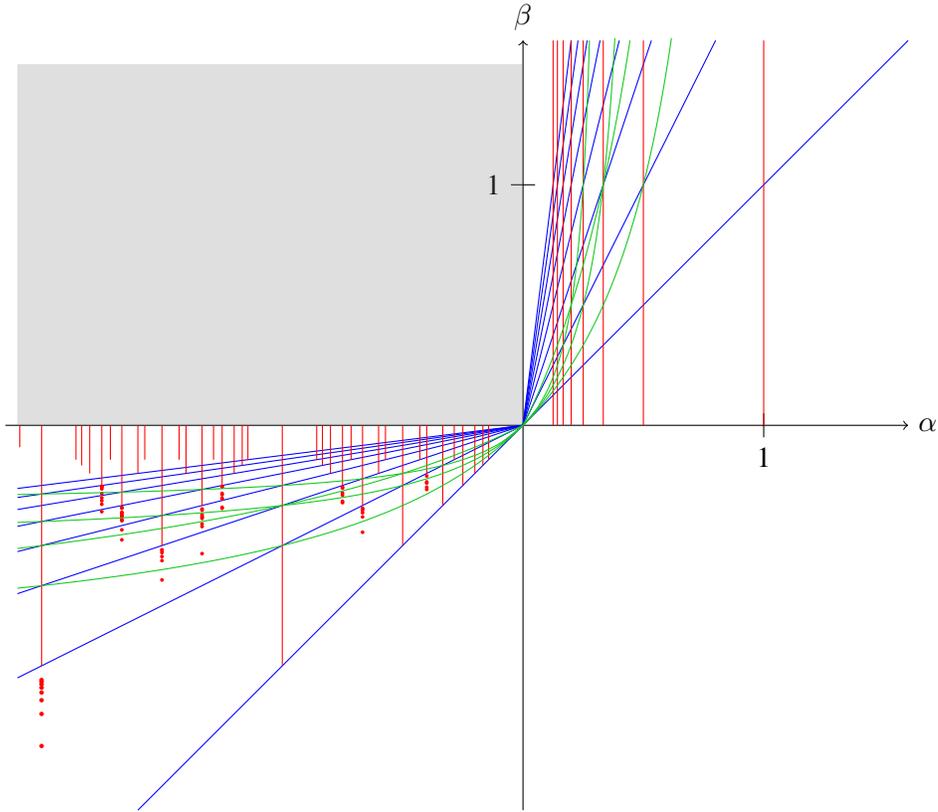

The structure of $S$  exhibits new phenomena  
compared to the set $S_0$ of commuting dilations characterized in 
Theorem \ref{thm:commuting}.  
We may recover  $S_0$ from the set $S$ by intersecting $S$ with its reflection across the line of slope one through the origin.

%
%

\subsection{The set  $S$ is closed}

We  deduce a topological property of the nonnegative commutator relation
from the  classification theorems.

\begin{thm}[Closed Set Property of $S$]
\label{thm:closed}
The set $S$ of all pairs of dilation factors $(\alpha, \beta)$ which satisfy the nonnegative commutator inequality ${[f_\alpha, f_\beta]\geq 0 }$,
is a closed subset of $\RR^2$.
\end{thm}

\noindent  The property that $S$ is a closed set is not unexpected but is also not obvious because the maps $f_{\alpha}, f_{\beta}$ are discontinuous functions of $x$. 
We deduce it only  as a consequence of  the complete classification of the solution set $S$ given in Theorems \ref{thm:positive} and \ref{thm:negative}.

%
%

\subsection{Preorder on $\RR^*$ induced by $S$}

A  significant  fact  established during the proofs is a transitivity property
of the nonnegative commutator relation, 
which  encodes  a compatibility property of commutators of three pairs of functions.
This property implies we have an induced preorder on the set of nonzero real numbers,
which is of independent interest.  
It was pointed out to us  by David Speyer as following from  our classification arguments. 

%
%
\begin{thm}[Nonnegative commutator transitivity]
\label{thm:poset} 
The  nonnegative commutator relation is transitive on non-zero dilated floor functions, meaning that for non-zero dilation factors $\alpha, \beta, \gamma$, 
\begin{equation}\label{eq:preorder1}
[f_\alpha, f_\beta] \geq 0 \quad\text{and}\quad [f_\beta, f_\gamma] \geq 0 
\qquad\text{imply that}\qquad [f_\alpha, f_\gamma] \geq 0 .
\end{equation}
\end{thm}

This transitivity property  is a special property of the class of dilated floor functions; 
in general  one can  find examples of nondecreasing functions
$f, g, h$ on the real line satisfying  $[f, g](x)  \ge 0$ and $[g, h](x) \ge 0$, 
for which  $[f, h](x) \ge 0$ does not hold.

It follows from transitivity that this relation determines a preorder on the set of nonzero real numbers.
This preorder is not a partial order; 
in particular all the elements
$\{ \frac{1}{n}: \, n \ge 1\}$ are equivalent in  this order. 
It induces a partial order by identifying  equivalent elements.
The induced partial order  extends the divisibility relation on positive integers.
The  elements in the nontrivial equivalence class are all positive, 
so the preorder is already a  partial order when restricted to negative values of $\alpha$.

The proof of the preorder property \eqref{eq:preorder1}
does not require establishing the detailed classification results in
Theorems \ref{thm:positive} and Theorem \ref{thm:negative}; it is established in Section \ref{sec:rounding}.

%
%
%
\section{Main results: positive dilations}
\label{sec:2}
\setcounter{equation}{0}

This paper determines when $[f_\alpha,f_\beta]\geq 0$ holds for positive dilations and mixed sign dilations.
It proves the closed set property for the solution set $S$ on the closed subset of $\RR^2$ where at
least one of $\alpha \ge 0$ or $\beta\ge 0$ holds. 
It proves the preorder theorem for all nonzero dilation parameters. 
The main part of the analysis concerns the positive dilations case,
where  both $\alpha, \beta>0$.
The remainder of this section states further  results for positive dilations,
and gives the organization of the rest of the paper.

%
%

\subsection{Geometric description of $S$: positive dilations} 
\label{subsec:21}

We give a more detailed geometric description of the solutions $S$ in
the positive  dilations case in $(\alpha, \beta)$ coordinates.

We picture the solution set $S$ for positive dilations
in $(\alpha, \beta)$-coordinates  in Figure  \ref{fig:ab-coord}.

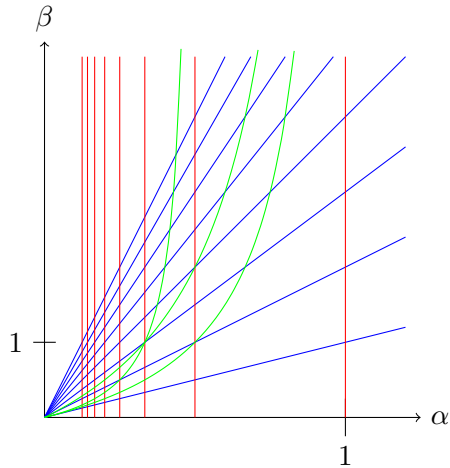
\begin{figure}[h]
\begin{center}
\begin{tikzpicture}[scale=.50]
  \draw[->] (0,0) -- (10,0) node[right] {$\alpha$};
  \draw[->] (0,0) -- (0,10) node[above] {$\beta$};
  \draw (8,0.5) -- (8,-0.5) node[below] {$1$};
  \draw (0.3,2) -- (-0.3,2) node[left] {$1$};
  \draw[xscale=8,yscale=2,domain=0:1.2,smooth,variable=\x, blue] plot ({\x},{\x});
  \draw[xscale=8,yscale=2,domain=0:2.4,smooth,variable=\x, blue] plot ({\x/2},{\x});
  \draw[xscale=8,yscale=2,domain=0:3.6,smooth,variable=\x, blue] plot ({\x/3},{\x});
  \draw[xscale=8,yscale=2,domain=0:4.8,smooth,variable=\x, blue] plot ({\x/4},{\x});
  \draw[xscale=8,yscale=2,domain=0:4.8,smooth,variable=\x, blue] plot ({\x/5},{\x});
  \draw[xscale=8,yscale=2,domain=0:4.8,smooth,variable=\x, blue] plot ({\x/6},{\x});
  \draw[xscale=8,yscale=2,domain=0:4.8,smooth,variable=\x, blue] plot ({\x/7},{\x});
  \draw[xscale=8,yscale=2,domain=0:4.8,smooth,variable=\x, blue] plot ({\x/8},{\x});
  \draw[xscale=8,yscale=2,domain=0:4.8,smooth,variable=\y, red]  plot ({1},{\y});
  \draw[xscale=8,yscale=2,domain=0:4.8,smooth,variable=\y, red]  plot ({1/2},{\y});
  \draw[xscale=8,yscale=2,domain=0:4.8,smooth,variable=\y, red]  plot ({1/3},{\y});
  \draw[xscale=8,yscale=2,domain=0:4.8,smooth,variable=\y, red]  plot ({1/4},{\y});
  \draw[xscale=8,yscale=2,domain=0:4.8,smooth,variable=\y, red]  plot ({1/5},{\y});
  \draw[xscale=8,yscale=2,domain=0:4.8,smooth,variable=\y, red]  plot ({1/6},{\y});
  \draw[xscale=8,yscale=2,domain=0:4.8,smooth,variable=\y, red]  plot ({1/7},{\y});
  \draw[xscale=8,yscale=2,domain=0:4.8,smooth,variable=\y, red]  plot ({1/8},{\y});
  \draw[xscale=8,yscale=2,domain=0:.83,smooth,variable=\x, green] plot ({\x},{\x/(1-\x)});
  \draw[xscale=8,yscale=2,domain=0:.71,smooth,variable=\x, green] plot ({\x},{2*\x/(1-\x)});
  \draw[xscale=8,yscale=2,domain=0:.454,smooth,variable=\x, green] plot ({\x},{\x/(1-2*\x)});
  \end{tikzpicture}
\end{center}
\caption{Positive dilation solutions of $S$  pictured in $(\alpha,\beta)$-coordinates}
\label{fig:ab-coord}
\end{figure} 

The positive dilation solutions  in  $S$ can be rewritten as  solutions to
\begin{equation}
\label{eq:positive2}
m \alpha +n \frac{\alpha}{\beta} =1.
\end{equation}
with $m, n\ge 0$. 
These solutions comprise three distinct families of curves:
\begin{enumerate}
\item[{\em Case (i-a)}] 
if $m = 0$, then
  \eqref{eq:positive2} cuts out  an oblique line through the origin with integer slope $n$; 
 \item[{\em Case (i-b)}] 
 if $n = 0$, 
 then \eqref{eq:positive2} cuts out  a vertical line  $\alpha=\frac{1}{m}$; 
\item[{\em Case (i-c)}] 
if both $m ,n\geq 1$, then \eqref{eq:positive2}
cuts out a rectangular hyperbola,
which approaches the origin with integer slope $n$,
where it is tangent to one of the $m=0$ solutions,
and which has  vertical asymptote $\alpha= \frac{1}{m}$,
so is tangent at infinity to one of the $n = 0$ solutions.
\end{enumerate}

%
%

\subsection{Symmetries of $S$: positive dilations} 

In the process of proving the classification theorems we 
establish several symmetries of the nonnegative commutator set $S$ in the positive dilation case.

%
%
\begin{thm}[Symmetries of the set $S$: positive dilations]
\label{thm:symmetries-p}
For positive dilations $\alpha, \beta >0$, the set $S$ is mapped into itself under the following symmetries:
\begin{enumerate}[(i)]
\item For any integer $m \ge 1$, if $(\alpha, \beta) \in S$, then $(m\alpha, \beta) \in S$.
\item For any integer $m \ge 1$, if $(\alpha, \beta) \in S$, then $(\frac{1}{m}\alpha, \frac{1}{m}\beta) \in S$.
\end{enumerate}
\end{thm}

There is an additional symmetry of a different type on the positive solutions of $S$,
under the birational transformation 
\begin{equation}\label{eqn:birat1}
(\alpha, \beta) \mapsto ( \frac{\alpha}{\beta}, \frac{1}{\beta}).
\end{equation}
This  birational  transformation acts as an involution on the open first quadrant of the plane.

%
%
\begin{thm}[First quadrant birational symmetry of the set $S$]
\label{thm:symmetries2}
On the region of  positive dilations,
the  set $S$ is invariant under the birational symmetry \eqref{eqn:birat1}.
That is, if  $\alpha, \beta >0$ and $(\alpha, \beta) \in S$, then 
$( \frac{\alpha}{\beta}, \frac{1}{\beta}) \in S$.
\end{thm}

The set $S$ viewed under a certain change of coordinates is related to disjointness of Beatty sequences, 
as we  discuss next. 
In terms of the change of variables 
$(\um,\vm)= (\frac{1}{\alpha}, \frac{\beta}{\alpha})$, 
the birational transformation \eqref{eqn:birat1}
becomes transposition $(\um,\vm) \mapsto (\vm,\um)$.

%
%
\subsection{Relation to disjoint Beatty sequences: positive dilations}
\label{subsec:beatty}
The  problem of classifying  which positive values $(\alpha, \beta)$ correspond to
a nonnegative commutator $[f_\alpha,f_\beta]$ has  a  close parallel with
the problem of classifying disjoint Beatty sequences, as we explain below. 

We recall the definition of Beatty sequences.
Let $\NN^{+} = \{1,2,\ldots\}$ denote the positive integers. 

\begin{defi}\label{def:44}
Given a positive real number $u$, the {\em Beatty sequence $\sB^{+}(u)$ on $\NN^{+}$} 
 is the set
\[
\sB^{+}(u) := \{ \floor{ n u }: \,  n \in \NN^+\} .
\]
\end{defi}

This definition was motivated by  a problem posed
in the American Mathematical  Monthly 
by Samuel Beatty \cite{Bea26} in 1926, which may be re-stated in the following form.

\begin{thm}[``Beatty's Theorem'']
\label{thm:20} 
Let $u, v$ be positive irrational numbers 
satisfying 
$$
\frac{1}{u} + \frac{1}{v} =1.
$$
Then the Beatty sequences $\sB^{+}(u)$ and $\sB^{+}(v)$
partition the positive integers $\NN^{+}$, i.e.
$$
\sB^{+}(u) \cap \sB^{+}(v) = \emptyset \quad \mbox{and} \quad \sB^{+}(u) \cup \sB^{+}(v) = \NN^{+}.
$$
\end{thm}

Beatty's problem was solved in 1927  by L. Ostrowski and J. Hyslop, 
and by A. C. Aitken (and not by Beatty),  see \cite{Bea26}. 
The result was  also obtained independently in 1927 by Uspensky \cite{Usp27}.
Moreover,  it  had already been  noted in 1894 by Lord Rayleigh \cite[p.122]{Rayleigh:1894}
in the context of a vibrating string.
A converse to Theorem \ref{thm:20} holds, explicitly noted by Fraenkel \cite[p. 6]{Fra69}:
 the sets $\sB^{+}(u)$ and $\sB^{+}(v)$ partition the
set $\NN^{+}$ only if $u, v$ are irrational and satisfy $\frac{1}{u} + \frac{1}{v} =1$.

The more general question  ``When are two Beatty sequences $\sB^{+}(u)$ and $\sB^{+}(v)$ disjoint?''
was posed and answered by Skolem \cite[Theorem 8]{Sko57} in 1957.
A proof  was given in a  follow-up 1957 result of Bang \cite[Theorem 9]{Ban57},
and another proof was given in  Niven \cite[Theorem 3.11]{Niven:63}.
For  further  generalizations of Beatty's
partition theorem for $\NN^{+}$, consult O'Bryant \cite{OB03}.

\begin{prop}[Skolem-Bang]
\label{prop:disjoint-Beatty}
For $u, v >0$ the Beatty sequences $\sB^{+}(u)$, $\sB^{+}(v)$ are disjoint 
   if and only if both $u$ and $v$ are irrational, and there exist integers $m,n \geq 1$ such that
\begin{equation}\label{eqn:Beatty-disjoint}
\frac{m}{u} + \frac{n}{v} = 1.
\end{equation}
\end{prop}

This result has a strong parallel with the 
conclusion of Theorem \ref{thm:positive}, 
after making the
change of variables $(\um,\vm)= (\frac{1}{\alpha}, \frac{\beta}{\alpha})$.
We call these new variables {\em $(\um,\vm)$-coordinates} for the nonnegative commutator problem.
(The inverse change of variables is $(\alpha, \beta) = (\frac{1}{\um}, \frac{\vm}{\um})$.)

\begin{thm}[Theorem \ref{thm:positive} in $(\um,\vm)$-coordinates]
\label{thm:positive-uv}
Given   $\um, \vm > 0$, the inequality
\[ 
\floor{\frac{1}{\um} \floor{\frac{\vm}{\um} x}} \geq \floor{\frac{\vm}{\um} \floor{\frac{1}{\um} x}} \quad\text{for all }x\in\RR
\]
holds if and only if 
there are integers $m, n\geq 0$, not both zero, such that 
\begin{equation}\label{eq:positive-uv}
\frac{m}{\um} + \frac{n}{\vm} = 1.
\end{equation}
\end{thm}

The criterion \eqref{eq:positive-uv} of this theorem is similar to the criterion 
\eqref{eqn:Beatty-disjoint} in the disjoint Beatty sequences result of Skolem and Bang. 
There are some differences: 
the  criterion of Theorem \ref{thm:positive-uv}
admits a wider range of $m, n$  than  for disjoint Beatty sequences
and it lifts the irrationality restriction on the coordinates $(u,v)$. 

The similarity is no coincidence.
The circle of ideas used to prove Theorem \ref{thm:positive-uv}
permit us to prove a parallel Beatty sequence result, 
that applies to an extension of the notion of Beatty sequences from $\NN^{+}$ to $\ZZ$.

\begin{defi}\label{def:45}
Let $u$ be a positive real number.
\begin{enumerate}
\item
The {\em  Beatty sequence $\sB(u)$  over $\ZZ$} is 
$$
\sB(u) := \{ \floor {n u} : \,  n \in \ZZ \} = \{ \floor {x} : \,  x \in u\ZZ \} .
$$

\item 
The  {\em reduced  Beatty sequence $\sB_0(u)$  over $\ZZ$}  is
$$
\sB_0(u) := \{ \floor {n u} : \,  n \in \ZZ \, \mbox{and} \,\, nu \not\in \ZZ\} 
= \{ \floor{ x } : \,  x \in u\ZZ \smallsetminus \ZZ\}.
$$
\end{enumerate}
\end{defi}

All  Beatty sequences over $\ZZ$ contain $0 = \floor 0$. 
For  $0 < u <1$ one  has $\sB_0(u) = \sB(u)= \ZZ$. 
If  $u>1$ is irrational   then  the reduced Beatty sequence is almost equal to the full one:
$\sB_0(u) = \sB(u) \smallsetminus \{0\}.$ 
If $u=\frac{r}{s}>1$ is rational, given in lowest terms, then the reduced Beatty sequence
$\sB_0(u) = \sB(u) \smallsetminus r\ZZ.$ 
In particular if $u$ is an integer then $\sB_0(u) = \emptyset$, 
while if $u$ is not an integer, then $\sB_0(u)$ is an infinite set.

Using results of this paper, we now establish a 
criterion for disjointness of reduced (full) Beatty sequences, 
which parallels the Skolem-Bang result.

\begin{thm}[Disjointness of reduced Beatty sequences over $\ZZ$]
\label{thm:27}
For $u, v>0$ the reduced Beatty sequences $\sB_0(u)$ and $\sB_0(v)$ over $\ZZ$ are
disjoint if and only if there exist integers $m, n \ge 0$, not both zero,  such that
\begin{equation}\label{eq:positive-beatty-uv}
\frac{m}{u} + \frac{n}{v} = 1.
\end{equation}
\end{thm}

To deduce this result, we use   Proposition \ref{lem:diagonal-pos} of this paper, 
which asserts that for positive dilations ($\alpha, \beta >0$) 
the nonnegative commutator property
\[ 
\floor{\alpha \floor{\beta x}} \geq \floor{\beta \floor{\alpha x}} \quad \text{for all}\quad  x\in\RR
\]
holds if and only if  the quantities $\um =\frac{1}{\alpha}$ and $\vm = \frac{\beta}{\alpha}$ satisfy
\begin{equation*}
\sB_0(\um) \bigcap \sB_0(\vm) = \emptyset.
\end{equation*}
Theorem \ref{thm:27}  follows on combining this result
with  Theorem \ref{thm:positive-uv},  
together with the observation that the change of variable from $(\alpha, \beta)$-coordinates 
to $(\um,\vm)$-coordinates is 
a bijection from the positive quadrant of $\RR^2$ to itself.

The proof of Theorem \ref{thm:positive-uv}
requires a new ingredient not present  in establishing  criterion 
\eqref{eqn:Beatty-disjoint} for disjoint Beatty sequences, which is
needed to handle  rational $\um$ and $\vm$ satisfying \eqref{eq:positive-uv}.
It involves a relation  with the two-dimensional Diophantine Frobenius problem,
appearing in case (2-b) of Theorem \ref{thm:pos-necessary} in Section \ref{subsec:pos-necessary}.


%
%
\subsection{Outline of proofs}
\label{subsec:outline}

In Section \ref{sec:basic} we  determine when the commutator relation  
$[f_\alpha, f_\beta]\geq 0$  holds  
in the  case when the dilation factors $\alpha, \beta$ have different signs.

In Section \ref{sec:rounding} we introduce a family of rounding functions 
depending on a dilation factor $\alpha$,
which are variations of the floor function that all have average slope $1$.
This family of functions may be of independent interest. 
Using rounding functions, we  derive the criterion 
Proposition  \ref{prop:rounding-pos}
for  the nonnegative commutator relation.
We derive ordering relations among different rounding functions,
and from these derive the symmetries of $S$ on the positive dilation region
 (Theorem \ref{thm:symmetries-p}).
We also deduce the transitivity of the nonnegative commutator relation (Theorem \ref{thm:poset}),
which implies there is an induced preorder on the set of nonzero dilations.

In Sections \ref{sec:positive-S} and \ref{sec:positive-N} we prove Theorem \ref{thm:positive}.
We address the two directions of the ``if and only if'' statement of this theorem separately in 
Section \ref{sec:positive-S} (sufficiency of criterion \eqref{eq:positive})
and Section \ref{sec:positive-N} (necessity of \eqref{eq:positive}).
The specific statements are given in Theorem \ref{thm:pos-sufficient} and Theorem \ref{thm:pos-necessary}.
The two parts of the proof  use different coordinate systems for the dilation parameters.  
The symmetry of exchanging $\um$ and $\vm$ 
in the  $(\um,\vm)$-coordinate system for disjoint Beatty sequences
extends to a birational symmetry of $S$ in the $(\alpha, \beta)$-coordinates for positive dilations.

In Section \ref{sec:consequences}  we prove the closure  Theorem \ref{thm:closed}
for $S$ in the case where at least one parameter $\alpha$ or $\beta$ is nonnegative. 
The proof for the remaining  case of  Theorem \ref{thm:closed} 
in the negative dilations region is deferred to Part II \cite{LagR:2018b}.

%
%
\subsection{Notation}

(1) The moduli space parametrizing  dilations, also termed ``parameter space,''
 has  coordinates denoted $(\alpha, \beta)$.
The proofs use several different birationally transformed coordinate systems for parameter space.
The  coordinate system  $(\alpha, \beta)$ is used in Sections \ref{sec:basic} and \ref{sec:rounding},
 the  $(\um, \vm)$-coordinate system is used in Section \ref{sec:positive-S}, 
and   the $(\talpha, \trho)$ coordinate system is used in
Section \ref{sec:positive-N}. 
All of these coordinates are positive real-valued in Sections \ref{sec:positive-S} and \ref{sec:positive-N}.

(2) Variables denoted $x$ and $y$  will generally
refer to coordinates of  graphs of dilated floor functions and/or to
 graphs of commutator functions $[f_{\alpha}, f_{\beta}](x)$. 
 These coordinates are  (positive or negative) real-valued and we refer to them as 
``function coordinates.''

%
%
\section{Mixed sign dilations  and preorder theorem}
\label{sec:basic}
\setcounter{equation}{0}

In this section we consider the nonnegative commutator relation $[f_\alpha,f_\beta]\geq 0$ when the dilation factors $\alpha, \beta$ differ in sign.

\subsection{Floor function basics}
The floor function $\floor{x}$ is defined to be the greatest integer which is no larger than $x$,
while the ceiling function $\ceil{x}$ is defined to be the least integer which is no smaller than $x$.
In other words,
\begin{equation*}
\floor{x} = \max\{ n\in \ZZ : n\leq x \}, \qquad 
\ceil{x} = \min\{ n\in \ZZ : n\geq x\}.
\end{equation*}
We let $f_{\alpha}(x) = \floor {\alpha x }$ and $g_{\alpha}(x) = \ceil {\alpha x}$.

The following properties of the floor and ceiling functions  are immediate:
\begin{equation}\label{floor-bounds}
x-1 < \floor{x} \leq x, \qquad x \leq \ceil{x} < x+1 ;
\end{equation}
\begin{equation}\label{eq:monotonic}
x\leq y \quad\text{implies}\quad 
\floor{x} \leq \floor{y} \quad\text{and}\quad \ceil{x} \leq \ceil{y}.
\end{equation}
The floor and ceiling functions are conjugate under negation, i.e.
\begin{equation}\label{floor-ceil}
\ceil{x} = -\floor{-x}.
\end{equation}
Thus $\floor{-x} = -\ceil{x}$.

\subsection{Preorder property: proof of  Theorem  \ref{thm:poset}}
\label{subsec:partialorder}

We let  $\RR^*$ denote the nonzero real numbers, 
and let $(\RR^*, \preceq_C )$  denote the binary relation on $\RR^* $ defined by 
\[
\alpha \preceq_C \beta \quad \mbox{if} \quad  
[f_\alpha,f_\beta]  \geq 0 .
\]
This relation is reflexive, i.e. $\alpha\preceq_C \alpha$ for all $\alpha $.
By Theorem \ref{thm:mixed-sign},
we have $\alpha \preceq_C \beta$ whenever  $\alpha < 0 < \beta$.

We prove Theorem \ref{thm:poset}, which asserts that this relation is transitive,
and so defines a preorder on all nonzero  elements $\RR^{\ast}$.
We restate the theorem  in terms of $\preceq_C$ for the reader's convenience.

\begin{customthm}{\ref{thm:poset}}[Nonnegative commutator preorder]
The  nonnegative commutator relation  ${(\RR^*, \preceq_C )}$ is transitive, meaning that for non-zero dilation factors $\alpha, \beta, \gamma$, 
\begin{equation*}
[f_\alpha, f_\beta] \geq 0 \quad\text{and}\quad [f_\beta, f_\gamma] \geq 0 
\qquad\text{imply that}\qquad [f_\alpha, f_\gamma] \geq 0 .
\end{equation*}
\end{customthm}

\begin{proof}
Suppose $\alpha, \beta$ are nonzero. 
We observe that after the change of variable $x=\frac{y}{\alpha\beta}$, 
the nonnegative commutator relation
$[f_\alpha, f_\beta](x) \ge0$  
can be rewritten as
\begin{equation}\label{eq:rounding}
\floor{ \alpha \floor{ \frac{1}{\alpha} y}} \geq \floor{\beta \floor{\frac{1}{\beta} y}} \quad\text{for all }y\in \RR.
\end{equation}
This equivalent version of the nonnegative commutator relation \eqref{eqn:ineq}  has the important feature that the $\alpha$ and $\beta$ parameters are {\em separated} on the two sides of the inequality.

Now suppose  that $\alpha$, $\beta$, and $\gamma$ are non-zero 
such that $[f_\alpha,f_\beta]\geq 0$ and $[f_\beta,f_\gamma]\geq 0$.
Then by the criterion \eqref{eq:rounding}, for all $y\in \RR$ we have
$\floor{ \alpha \floor{ \frac{1}{\alpha} y}} \geq \floor{\beta \floor{\frac{1}{\beta} y}}$
and 
$\floor{ \beta \floor{ \frac{1}{ \beta}y}} \geq \floor{\gamma \floor{\frac{1}{\gamma} y}}$.
Thus
$\floor{ \alpha \floor{ \frac{1}{ \alpha} y}} \geq \floor{\gamma \floor{\frac{1}{ \gamma} y}}$
holds for all $y \in \RR$, whence $[f_\alpha, f_\gamma] \geq 0$.
\end{proof} 

%
%
\begin{rmk}\label{rmk:partial-order}
Theorem \ref{thm:poset}  establishes the preorder property without
determining its structure.
The fact that for $\alpha>0$  the partial order $\preceq_C$
contains the divisibility relation 
follows from the symmetries of $S$ described in Theorem  \ref{thm:symmetries-p}(i).
\end{rmk}

%
%
\subsection{Mixed sign dilations: proof of Theorem \ref{thm:mixed-sign}}
\label{subsec:mixed-sign}

We restate Theorem \ref{thm:mixed-sign} here for the convenience of the reader, and give a proof.

\begin{customthm}{1.1}[Mixed Sign Dilations Classification]
Suppose dilation factors $\alpha$ and $\beta$ have opposite signs.
\begin{enumerate}[(a)]
\item
 If $\alpha < 0$ and $\beta > 0$, then the commutator relation $[f_{\alpha}, f_\beta]\geq 0$ 
is  satisfied.
\item[(b)]
 If $\alpha > 0$ and $\beta < 0$, then the commutator relation $[f_\alpha, f_{\beta}]\geq 0$ 
is not satisfied.
\end{enumerate}
\end{customthm}

\begin{proof}
We set $\alphap = |\alpha|$ and $\betap = |\beta|$, so  that $\alphap, \betap > 0$.

(a) (Case $\alpha<0, \beta >0$) 
By setting $x = \frac1{\alpha\beta}y$, it suffices to show that $\floor{ \alpha \floor{ \frac{1}{\alpha} y}} \geq \floor{\beta \floor{\frac{1}{\beta} y}}$
for all $y\in \RR$.
Note that  $(\alpha, \beta) = (-\alphap, \betap)$ with $\alphap, \betap > 0$.
We have
\begin{align*}
\qquad 
{-\alphap \floor{\frac{1}{-\alphap} y}} 
&= {\alphap \ceil{\frac{1}{\alphap} y}}
    && \text{by }\eqref{floor-ceil},\\
&\geq y 
    && \text{by }\eqref{floor-bounds},\\
&\geq {\betap \floor{\frac{1}{\betap} y}} 
    && \text{by }\eqref{floor-bounds}. \qquad
\end{align*}
By \eqref{eq:monotonic}, this implies
$\floor{-\alphap \floor{\frac{1}{-\alphap} y}} \geq \floor{ \betap \floor{\frac1{\betap} y}}$
so the commutator relation $[f_{-\alphap}, f_{\betap}] \ge 0$ holds.

(b) (Case $\alpha>0, \beta <0$) 
Now we have $(\alpha, \beta) = (\alphap, -\betap)$.
By the above argument, for any $y$
$$
{ \floor{\alphap \floor{\frac{1}{\alphap}y}} \leq \floor{ -\betap \floor{\frac{1}{-\betap} y}} }.
$$
Moreover, this inequality is strict for $y = -\epsilon$ a negative number sufficiently close to zero, since
\begin{equation*}
\floor{ \alphap \floor{-\frac{1}{\alphap} \epsilon}}  = \floor{-\alphap} \leq -1 < 0 = \floor{ -\betap \floor{ \frac{1}{\betap} \epsilon}}.
\end{equation*}
In particular $[f_{\alphap}, f_{-\betap}](x) <0$ for $x = \frac{\epsilon}{\alphap\betap}$, so 
$[f_\alphap, f_{-\betap}]$ is not always nonnegative.
\end{proof}

%
%
\section{Rounding Functions}
\label{sec:rounding}
\setcounter{equation}{0}

To deal with the nonnegative commutator relation for  positive dilations, we 
introduce a new family of functions, which we call rounding functions. 

The dilated floor function $f_\alpha(x) = \floor{\alpha x}$ is a variation of the floor
function  which sends $\RR \to \ZZ$ but is modified to have average slope $\alpha$.
Rounding functions, on the other hand, are variations of the floor function which all have average slope $1$, but instead map $\RR \to \alpha \ZZ$ for some real parameter $\alpha$.

\begin{defi}
Given parameter $\alpha \neq 0$, 
let $\floor{x}_\alpha$ and $\ceil{x}_\alpha$ 
denote the {\em (lower) rounding function}  
defined by
\begin{equation*} 
\floor{x}_\alpha := \alpha \floor{\frac{1}{\alpha} x}
\end{equation*}
and the {\em (upper)
rounding function}  defined by
\begin{equation*} 
\ceil{x}_\alpha := \alpha \ceil{\frac1{\alpha} x}. 
\end{equation*}
\end{defi}

The usual floor and ceiling functions are 
contained in this family as $\floor{x} = \floor{x}_1$ and $\ceil{x} = \ceil{x}_1$, respectively.
The identity \eqref{floor-ceil} says that
$\ceil{x} = \floor{x}_{-1}$; 
since we allow  $\alpha$ to be negative, the above definition gives a single family of rounding functions, namely
\begin{equation}\label{eqn:floor-ceil}
\ceil{x}_{\alpha} = \floor{x}_{-\alpha}.
\end{equation}
Furthermore it is natural to extend the definition to $\alpha=0$ by setting $\floor{x}_0 = x$, 
observing that $x$
 is the  limit of $\floor{x}_{\alpha}$ as $\alpha $ approaches zero from either direction.
This family of functions seems of some interest in its own right. 
We also use the notation $r_\alpha(x) = \floor{x}_\alpha$ for the rounding function when it is more convenient.

Rounding functions appear in the separation of variables criterion \eqref{eq:rounding}
for a nonnegative commutator, for all nonzero $\alpha, \beta$.  
That criterion  can be rewritten in terms of  rounding functions as
\begin{equation}\label{eqn:rounding2}
\floor{ \floor{y}_\alpha}_1 \ge \floor{\floor{y}_\beta}_1  \quad \mbox{ for all} \quad y \in \RR.
\end{equation}
A related family of functions, the {\em strict rounding functions},
will be required  in the negative dilations case treated in Part II \cite{LagR:2018b}.


\subsection{Rounding functions: ordering inequalities} 

We  classify when two rounding functions are comparable 
in the sense that the graph of one lies (weakly) below the other.
These relations will be used to analyze  symmetries of the set $S$ of dilation factors 
$(\alpha,\beta)$ which satisfy ${ [f_\alpha,f_\beta]\geq 0.}$

\begin{prop}[Rounding Function Ordering Inequalities]
\label{prop:rounding-relations}
Suppose dilation factors $\alpha,\beta $ are positive. 
\begin{enumerate}[(a)]
\item 
xThe inequality $\floor{x}_\alpha \leq \floor{x}_\beta$ holds for all $x\in\RR$ if and only 
if $\alpha = m\beta$ for some integer $m\geq 1$.
\item
The inequality $\ceil{x}_\alpha \leq \ceil{x}_\beta$, 
holds for all $x\in\RR$ 
if and only if $\beta = m\alpha$ for some integer $m\geq 1$.
\end{enumerate}
\end{prop}


\begin{proof}
(a) By definition of the  rounding functions,  after multiplication by $\frac1{\beta}$ 
the inequality $\floor{x}_\alpha \leq \floor{x}_\beta$ is equivalent to the inequality
 $\floor{\frac1{\beta}x}_{\alpha/\beta} \leq \floor{\frac1{\beta}x}_1$.
Thus it suffices to show that the condition
\begin{equation*}
\floor{x}_\alpha \leq \floor{x}_1 = \floor{x} \quad\text{for all } x\in \RR
\end{equation*}
holds if and only if $\alpha$ is a positive integer.
If $\alpha=m$ is a positive integer, 
then $\floor{x}_m$ is an integer no larger than $x$ so $\floor{x}_m \leq \floor{x}$
by definition of the floor function.
Conversely if $\alpha$ is not an integer, then at $x = \alpha$ we have 
$\floor{\alpha}_\alpha = \alpha > \floor{\alpha}$.

(b) This inequality follows from (a) 
 by conjugating the real line $\RR$ with respect to negation $x \mapsto -x$
(i.e. ``rotating the graph by $180^\circ$'').
Namely, since $\ceil{x}_\alpha = -\floor{-x}_\alpha$ the given condition is equivalent to 
\begin{equation*}
-\floor{x}_\alpha \leq -\floor{x}_\beta \quad\text{for all } x\in \RR.
\end{equation*}
Multiplying by $-1$ switches the direction of the inequality and reduces to (a). 
\end{proof}

\subsection{Rounding function criterion: positive dilations}
\label{subsec:pos-rounding-crit}

We first give a condition in terms of  rounding functions 
which is equivalent to the nonnegative commutator condition 
$[f_\alpha, f_\beta]\geq 0$ on dilated floor functions $f_\alpha(x) = \floor{\alpha x}$.

\begin{prop}[Nonnegative Commutator Relation: Rounding Function Criteria]
\label{prop:rounding-pos}
For $\alpha, \beta > 0$, the following properties are equivalent.
\begin{enumerate}
\item[(R1)]  The  nonnegative commutator relation holds:
\[ 
[f_\alpha, f_\beta](x) \geq 0  \quad \mbox{for all}\quad x \in \RR.
\]

\item[(R2)] (Upper rounding function) There holds
\begin{equation}\label{eqn:criterion-positive}
\ceil{n}_\alpha \leq \ceil{n}_\beta \quad\text{for all} \quad n \in \ZZ,
\end{equation}
where $\ceil{x}_\alpha = \alpha \ceil{\frac{1}{\alpha} x}.$
\end{enumerate}
\end{prop}

\begin{rmk}
\noindent The  condition  (R2) can be rewritten 
in terms of lower rounding functions, via \eqref{eqn:floor-ceil},  as:
\begin{enumerate}
\item[{\em (R3)}]
{\em (Lower rounding function) There holds
\begin{equation}\label{eqn:criterion:lower}
\floor{n}_{-\alpha} \leq \floor{n}_{-\beta} \quad\text{for all }n\in \ZZ.
\end{equation}
}
\end{enumerate}
\end{rmk}

\begin{proof}[Proof of Proposition \ref{prop:rounding-pos}]
The {\em upper level set} $U[f](y)$  of a real-valued function $f$ at level $y$ is:
$$
U[f](y) := \{ x\in \RR : f(x) \ge y \}.
$$
For  real-valued functions $f$ and $g$ on $\RR$, to prove that $f(x) \ge g(x)$ holds for all $x$, it suffices to show the set inclusions for corresponding upper level sets 
\begin{equation}\label{eq:level-inclusion}
U[g](y) \subset U[f](y) \quad \mbox{holds for all $y$ in the range of $g$}.
\end{equation}

For positive $\alpha, \beta$ the  nonnegative commutator condition $[f_\alpha, f_\beta]\geq 0$
 is equivalent to  the condition 
 \begin{equation}\label{eq:ineq-nn}
\floor{\alpha \floor{\frac{1}{\alpha} x}} \geq \floor{\beta \floor{\frac{1}{\beta} x}}  
\quad\text{for all} \,\, x\in \RR.
\end{equation} 
The  functions $r_1\circ r_{\alpha}(x)= \floor{\alpha \floor{\frac{1}{\alpha} x}}$ and
 $r_1\circ r_{\beta}(x)= \floor{\beta \floor{\frac{1}{\beta} x}}$   
 both have range taking values in integers.   
Applying the level set criterion above to
 $f=r_1\circ r_{\alpha}$ and
 $g=r_1\circ r_{\beta}$, 
we   conclude that \eqref{eq:ineq-nn} 
holds if and only if 
\begin{equation}\label{eq:level-inclusion-rounding}
U [r_1\circ r_\beta](n)\subset U [r_1\circ r_\alpha](n) \quad \mbox{for all  $n \in \ZZ$.}
\end{equation}
The equivalence (R1) $\Leftrightarrow$ (R2) follows as  a consequence of 
the following  formula for the upper level sets: 
For all $\alpha > 0$,  
\begin{equation}\label{eq:level-value}
U[r_1\circ r_{\alpha}](n) = \{ x : x \ge \ceil{n}_{\alpha}\}.
\end{equation}
This formula follows by the chain of equivalences:
\begin{align*}
\qquad  \floor{ \alpha \floor{ \frac1\alpha x}}  \geq n
& \quad\Leftrightarrow\quad  \alpha \floor{ \frac1\alpha x}  \geq n 
    && \quad \mbox{(the right side is in} \,\, \ZZ \,\mbox{)}\\
& \quad\Leftrightarrow\quad    \floor{ \frac1\alpha x} \geq \frac{1}{\alpha} n  
    && \quad \mbox{(since $\alpha > 0$)}\\\
& \quad\Leftrightarrow\quad   \floor{ \frac1\alpha x} \geq \ceil{ \frac{1}{\alpha} n }   
    &&  \quad \mbox{(the left side is in} \,\, \ZZ \,\mbox{)} \\
& \quad\Leftrightarrow\quad   \frac1\alpha x \geq \ceil{ \frac{1}{\alpha} n }   
    && \quad \mbox{(the right side is in} \,\, \ZZ \,\mbox{)}\\
&\quad\Leftrightarrow\quad   x \geq \alpha \ceil{ \frac{1}{\alpha} n }  = \ceil{ n}_{\alpha} 
    && \quad \mbox{(since $\alpha >0$)}.
\end{align*} 
The level set inclusion \eqref{eq:level-inclusion-rounding}
is equivalent  by \eqref{eq:level-value} to the condition
\[ 
\lceil n\rceil_{\alpha} \le \lceil n\rceil_{\beta} \quad \mbox{holds for all $n \in \ZZ$},
\]
as asserted. 
\end{proof}

%
%
\subsection{Symmetries of $S$ for positive dilations: Proof of Theorem \ref{thm:symmetries-p}} 
\label{subsec:symmetries-pos}

We  deduce symmetries of the set $S$ for positive dilations from symmetries of
the integer rounding criterion given in  Proposition \ref{prop:rounding-pos}. 

\begin{proof}[Proof of Theorem \ref{thm:symmetries-p}]
We suppose $\alpha >0, \beta >0$ and are to show:
\begin{enumerate}[(i)]
\item for any integer $m\geq 1$,  if $(\alpha, \beta) \in S$  then $( \alpha, {m}\beta)\in S$.
\item for any integer $m\geq 1$,  if $(\alpha, \beta) \in S$  then $(\frac{1}{m}{\alpha}, \frac{1}{m}\beta) \in S$.
\end{enumerate}
By  Proposition \ref{prop:rounding-pos}  we have $(\alpha, \beta) \in S$ if and only if 
these parameters satisfy the rounding function  inequalities
\begin{equation}\label{rounding-positive}
\ceil{n}_\alpha \leq \ceil{n}_\beta \quad\text{for all }n\in\ZZ .
\end{equation}
Therefore, given that $(\alpha,\beta)$ satisfies \eqref{rounding-positive} it suffices to show that
\begin{enumerate}[(i)]
\item for any integer $m\geq 1$,  the dilation factors $( \alpha, {m}\beta)$ also satisfy  \eqref{rounding-positive}.
\item  for any integer $m\geq 1$,  the dilation factors $(\frac{1}{m}{\alpha}, \frac{1}{m}\beta)$ also satisfy
 \eqref{rounding-positive}.
\end{enumerate}
To show these, note that for any integer $m \ge 1$,

(i)  Proposition \ref{prop:rounding-relations} implies $\ceil{x}_\beta \leq \ceil{x}_{m\beta}$ for all $x\in\RR$, so in particular $\ceil{n}_\beta \leq \ceil{n}_{m\beta}$.

(ii) By definition of the rounding functions,
 $\ceil{n}_{\alpha/m} \leq \ceil{n}_{\beta/m} $ is equivalent to $\ceil{mn}_\alpha \leq \ceil{mn}_\beta$.
\end{proof}

%
%
\section{Positive Dilations Classification: Sufficiency}
\label{sec:positive-S}
\setcounter{equation}{0}

In this section we prove the sufficiency of   the criterion of Theorem \ref{thm:positive}
for  membership  $(\alpha, \beta) \in S$ of a pair of positive  dilations.
Namely, we show that if dilation parameters $\alpha, \beta$ satisfy the condition
\begin{equation}
\tag{\ref{eq:positive2}} 
m\alpha + n\frac{\alpha}{\beta} = 1 \quad\text{for some integers }m,n\geq 0
\end{equation}
then they also satisfy the nonnegative commutator relation $[f_\alpha, f_\beta]\geq 0$.

To do so it is convenient to make a birational  change of coordinates
of the parameter space describing the two dilations.
We map  $(\alpha, \beta)$ coordinates of parameter space to $(\um,\vm)$-coordinates, given by
\begin{equation}\label{eqn:uv-coords}
(\um,\vm) := (\frac{1}{\alpha}, \frac{\beta}{\alpha}).
\end{equation}
 The  map $(\alpha, \beta) \mapsto (\frac{1}{\alpha}, \frac{\beta}{\alpha})$ is an 
 involution sending the open first quadrant to itself, 
 hence in the other direction we have $(\alpha, \beta) = (\frac{1}{\um}, \frac{\vm}{\um})$.

The positive dilation part of the set $S$ is pictured in $(\um,\vm)$-coordinates in Figure \ref{fig:uv-coord}.

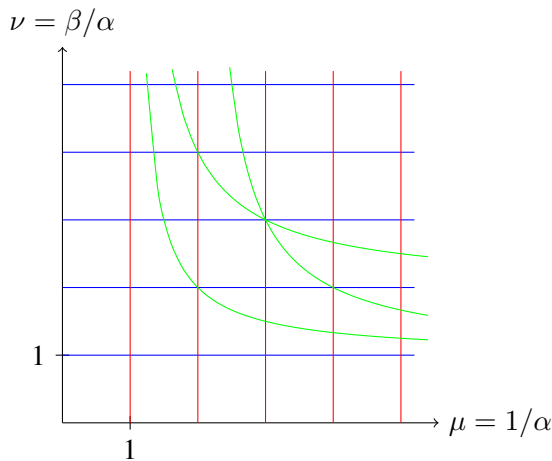
\begin{figure}[h]
\begin{center}
\begin{tikzpicture}[scale=.50]
  \draw[->] (0,0) -- (10,0) node[right] {$\um=1/\alpha$};
  \draw[->] (0,0) -- (0,10) node[above] {$\vm=\beta/\alpha$};
  \draw[scale=1.8] (1,0.1) -- (1,-0.1) node[below] {1};
  \draw[scale=1.8] (0.1,1) -- (-0.1,1) node[left] {1};
  
  \draw[scale=1.8,domain=0:5.2,smooth,variable=\x, blue] plot ({\x},{1});
  \draw[scale=1.8,domain=0:5.2,smooth,variable=\x, blue] plot ({\x},{2});
  \draw[scale=1.8,domain=0:5.2,smooth,variable=\x, blue] plot ({\x},{3});
  \draw[scale=1.8,domain=0:5.2,smooth,variable=\x, blue] plot ({\x},{4});
  \draw[scale=1.8,domain=0:5.2,smooth,variable=\x, blue] plot ({\x},{5});
  \draw[scale=1.8,domain=0:5.2,smooth,variable=\y, red]  plot ({1},{\y});
  \draw[scale=1.8,domain=0:5.2,smooth,variable=\y, red]  plot ({2},{\y});
  \draw[scale=1.8,domain=0:5.2,smooth,variable=\y, red]  plot ({3},{\y});
  \draw[scale=1.8,domain=0:5.2,smooth,variable=\y, red]  plot ({4},{\y});
  \draw[scale=1.8,domain=0:5.2,smooth,variable=\y, red]  plot ({5},{\y});
  \draw[scale=1.8,domain=1.24:5.4,smooth,variable=\x, green] plot ({\x},{\x/(\x-1)});
  \draw[scale=1.8,domain=1.62:5.4,smooth,variable=\x, green] plot ({\x},{2*\x/(\x-1)});
  \draw[scale=1.8,domain=2.47:5.4,smooth,variable=\x, green] plot ({\x},{\x/(\x-2)});
\end{tikzpicture}
\end{center}
\caption{Positive dilation solutions in $S$  in  $(\um,\vm)$-coordinates:
$\um = \frac{1}{\alpha}, \, \vm= \frac{\beta}{\alpha}$.}
\label{fig:uv-coord}
\end{figure}

In the new coordinates, Theorem \ref{thm:symmetries-p} says that 
solutions to $[f_{1/\um}, f_{\vm/\um}] \geq 0$ are  preserved under the maps
\[ 
(\um, \vm) \mapsto (m \um, \vm) \quad\text{and}\quad 
(\um, \vm)\mapsto (\um, m\vm) 
\]
for any integer $m\geq 1$.
These  symmetries of the nonnegative commutator condition \eqref{eqn:ineq} are  visually apparent after this  coordinate change.

%
%
\subsection{Lattice disjointness criterion: positive dilations}
\label{subsec:lattice-disjoint}

{ 
The  nonnegative commutator relation has a convenient reformulation in terms of the new parameters 
$(\um, \vm)$ as follows. 
It is a geometric criterion which involves  a  disjointness property
of the rectangular  lattice $\Lambda_{\um, \vm}= \um \ZZ \times \vm \ZZ$ in $\RR^2$ 
from an ``enlarged diagonal set''
$\mathcal{D}$.

\begin{defi}
The  {\em enlarged diagonal set } 
$\cD$ in $\RR^2$ is  the region
\begin{align}\label{def:regionD}
\cD :=& \bigcup_{n \in \ZZ} \{ (x,y) : n < x < n+1 \quad \mbox{and} \quad  n < y < n+1 \} \nonumber \\
 =&  \bigcup_{n \in \ZZ} \{ x \in (n, n+1)\} \times \{ y \in (n, n+1)\}.
\end{align}
\end{defi}

The  region $\mathcal{D}$ 
is a disjoint union of open unit squares. 
  Each open unit square  is symmetric by reflection across the diagonal $(x,y) \mapsto (y,x)$.
We can also characterize this set in terms of rounding functions, by 
\begin{align} 
\cD &= \{ (x,y) : \floor{x} < y < \ceil{x} \} 
= \{ (x,y) : \floor{y} < x < \ceil{y} \}. 
\end{align}
The set $\cD$ is pictured in Figure \ref{fig:diagonal-pos} below.

\begin{figure}[h]
\begin{center}
\begin{tikzpicture}[scale=1]
  \draw[->] (2,1) -- (5.2,1) node[right] {$x$};
  \draw[->] (1,2) -- (1,5.2) node[above] {$y$};
  \draw[-] (-0.2,1) -- (0,1); 
  \draw[-] (1,-0.2) -- (1,0);
  \draw[scale=1] (2,1.1) -- (2,1-0.1) node[below] {1};
  \draw[scale=1] (1.1,2) -- (1-0.1,2) node[left] {1};
  
  \foreach \n in {0,1,2,3,4} {
    \fill[gray,nearly transparent] (\n,\n) -- (\n,\n+1) -- (\n+1,\n+1) -- (\n+1,\n) -- cycle;
    \draw[dashed] (\n,\n) -- (\n,\n+1);
    \draw[dashed] (\n,\n+1) -- (\n+1,\n+1);
    \draw[dashed] (\n+1,\n+1) -- (\n+1,\n);
    \draw[dashed] (\n+1,\n) -- (\n,\n);
  }
\end{tikzpicture}
\end{center}
\caption{Region $\cD= \{\floor{y}<x<\ceil{y}\}$ of $\RR^2$, in gray.}
\label{fig:diagonal-pos}
\end{figure}

We also relate this geometric criterion to a disjoint reduced Beatty sequence criterion,
given in (P2) below. 
Recall from Section \ref{subsec:beatty} that for $u >0$ the reduced  Beatty sequence 
$\sB_0(u)$  in $\ZZ$ is given by 
$$
\sB_0(u) := \{ \floor x : \,  x \in u\ZZ \smallsetminus \ZZ\} .
$$

\begin{prop}[Nonnegative Commutator Relation: Lattice Disjointness Criterion]
\label{lem:diagonal-pos}

For $\um,\vm> 0$, the following three properties are equivalent.
\begin{enumerate}
\item[(P1)] The nonnegative commutator relation holds:
$$
[f_{1/\um}, f_{\vm/\um}](x) \geq 0\quad \mbox{for all} \quad x \in \RR.
$$ 

\item[(P2)]  The two-dimensional rectangular lattice
$\Lambda_{\um,\vm} := \ZZ \um \times \ZZ \vm = \{ (m\um, n\vm) : m,n\in\ZZ\} $
in $\RR^2$ is disjoint from the enlarged diagonal region $\mathcal{D}= \{(x, y): \floor{y} < x < \ceil{y}\}$.
 That is,  
\begin{equation}\label{eq:diag}
 \Lambda_{\um,\vm} \bigcap \mathcal{D} = \emptyset.
\end{equation}
 
\item[(P3)] The reduced Beatty sequences $\sB_0(\um)$ and $\sB_0(\vm)$ in $\ZZ$ are disjoint. 
That is,
\begin{equation}
\sB_0(\um) \bigcap \sB_0(\vm) = \emptyset.
\end{equation}
\end{enumerate}
\end{prop}


\begin{proof} 
To establish $(P1) \Leftrightarrow  (P2)$  we prove the contrapositive.
By Proposition \ref{prop:rounding-pos}, the commutator inequality $[f_{1/\um}, f_{\vm/\um}] \geq 0$ 
fails if and only if 
\begin{equation*}
\ceil{n}_{1/\um} > \ceil{n}_{\vm/\um} \quad\text{for some }n\in\ZZ.
\end{equation*}
By multiplying this inequality by $\um$ and expanding $\ceil{n}_{1/\um} = \frac{1}{\um}\ceil{\um  n}_1$, 
we get the equivalent condition
\begin{equation}\label{eq:rounding-pos2}
\ceil{n \um}_{1} > \ceil{n \um}_{\vm} \quad\text{for some }n\in\ZZ.
\end{equation}
If true, this means that the $\vm$-multiple $\ceil{n \um}_{\vm}$ lies in the half-open interval 
$\{x \in \RR : n \um \leq x < \ceil{n\um}\}$,
or in other words, 
\begin{equation}\label{eq:disjoint-pos1}
\text{there exist } m, n \in \ZZ \text{ such that }  n \um \leq m \vm < \ceil{n \um}
\end{equation}
where we pick the integer $m$ to satisfy $\ceil{n \um}_{\vm} = m \vm$.
Conversely, \eqref{eq:disjoint-pos1} implies \eqref{eq:rounding-pos2} so these conditions are equivalent.

Now consider the lattice $ \Lambda_{\um,\vm} = \{ (m \um, n \vm) : m,n\in\ZZ\} \subset \RR^2$.
Condition \eqref{eq:disjoint-pos1} says:
$$ 
[f_{1/\um}, f_{\vm/\um}] \geq 0 \,\, \mbox{fails} \quad \Leftrightarrow \quad
\Lambda_{\um,\vm} \, \mbox{intersects the region} \,\, 
\cD_0 = \{ (x,y) :  x \leq y < \ceil{x} \} \subset \RR^2
$$
Note that $\cD_0$ is  the ``upper half'' of the region $\cD = \{ (x,y) : \floor{x} < y < \ceil{x} \}$.
The ``lower half'' is the image of $\cD_0$ under the $180^\circ$ rotation $(x,y) \mapsto (-x,-y)$. 
If we let $-\cD_0$ denote the image of $\cD_0$ under this rotation, then $\cD = \cD_0\cup(-\cD_0)$.
Since the lattice $\Lambda_{u,v}$ is fixed by the rotation $(x,y) \mapsto (-x,-y)$, we have
\begin{align*}
\Lambda_{\um,\vm} \text{ intersects } \cD_0 &\quad\Leftrightarrow\quad 
 -\Lambda_{\um,\vm}=\Lambda_{\um,\vm} \text{ intersects } -\cD_0\\
 &\quad\Leftrightarrow\quad 
 \Lambda_{\um,\vm} \text{ intersects } (-\cD_0)\cup \cD_0 = \cD.
\end{align*}
The contrapositive is proved. 

To establish $(P2) \Leftrightarrow  (P3)$, we observe that
\begin{align*}
 \Lambda_{\um,\vm} \bigcap \mathcal{D} = \emptyset  &\quad\Leftrightarrow\quad
  \text{ for all } (k,\ell) \in \ZZ^2, \quad  (k \um, \ell \vm) \not\in \cD\\
 &\quad\Leftrightarrow\quad  \text{ if } \floor{k \um}=m  \text{ and }  \floor{\ell \vm}= m,  \text{ then } 
  k \um \in \ZZ \, \text{ or } \ell \vm \in \ZZ \\
 &\quad\Leftrightarrow\quad  \text{ for all }m \in \ZZ,\quad m\not\in \sB_0(\um) \bigcap \sB_0(\vm) \\
 &\quad\Leftrightarrow\quad \,\sB_0(\um) \bigcap \sB_0(\vm)  =\emptyset,
\end{align*}
as required.
\end{proof}

%
%
\subsection{Birational symmetry property for positive dilations: Proof of Theorem \ref{thm:symmetries2}}

Given the criterion of Proposition \ref{lem:diagonal-pos}, the symmetries of the nonnegative commutator relation may be deduced from the symmetries of $\cD$.
Recall that the region $\mathcal{D}$ is symmetric by reflection across the diagonal $(x,y) \mapsto (y,x)$.
Combining this observation with Proposition \ref{lem:diagonal-pos} yields
an extra symmetry of the commutator relation for
 positive dilation factors.

\begin{lem}[Birational  Symmetry Property of $S$]
\label{lem:symm-prop}
Suppose $\um, \vm >0$. Then 
\begin{enumerate}
\item Given parameters $\um,\vm>0$,  the nonnegative commutator relation 
$[f_{1/\um}, f_{\vm/\um}] \geq 0 $ holds
if and only if $ [f_{1/\vm}, f_{\um/\vm}] \geq 0$ holds.

\item The positive dilation part of the set $S$ viewed in $(\um,\vm)$-coordinates is invariant under the interchange of $\um$ and $\vm$.
\end{enumerate}
\end{lem}

\begin{proof}
(1) Interchanging $\um$ and $\vm$ has the effect of reflecting the lattice 
$\Lambda_{\um,\vm}$ across the line of slope 1 through the origin. The enlarged diagonal region $\cD$ is preserved under this reflection, so disjointness of 
$\Lambda_{\um,\vm}$ and $\cD$ is preserved under interchanging $\um$ and $\vm$.
By Proposition \ref{lem:diagonal-pos}, this disjointness condition (P2) is equivalent to membership in $S$ (P1).

(2) This restates (1), using the definition of the nonnegative commutator set $S$. 
\end{proof}

\begin{proof}[Proof of Theorem \ref{thm:symmetries2}]
We are to show that if $\alpha, \beta >0$ and $(\alpha,\beta)\in S$, then $(\frac1\beta, \frac{\alpha}{\beta}) \in S$.

Written in  $(\alpha, \beta)$-coordinates,   Lemma  \ref{lem:symm-prop}  (1) gives the result. 
Namely, for  $\alpha, \beta >0$  set
  $\alpha= \frac{1}{\um}$ and  $\beta= \frac{\vm}{\um}$, 
whence  $\frac{1}{\vm} = \frac{\alpha}{\beta}$ and $\frac{\um}{\vm} = \frac{1}{\beta}$.
Lemma  \ref{lem:symm-prop}  (1) concludes that
if   $(\alpha, \beta) = (\frac{1}{\um},  \frac{\vm}{\um}) \in S$, 
then $(\frac{\alpha}{\beta} , \frac{1}{\beta}) = (\frac{1}{\vm}, \frac{\um}{\vm}) \in S$.
\end{proof}

%
%
\subsection{Proof of sufficiency in Theorem \ref{thm:positive}}
\label{subsec:pos-sufficient}
The  sufficiency  direction of Theorem \ref{thm:positive} for positive dilations $\alpha, \beta>0$
asserts that  $(\alpha, \beta) \in S$  
whenever $m\alpha + n \frac{\alpha}{\beta} =1$ for some integers $m, n \ge 0$.
Restated in $(\um, \vm)$-coordinates, this is equivalent to the following proposition.

\begin{thm}[Sufficiency condition in $(\um, \vm)$-coordinates]
\label{thm:pos-sufficient}
Given parameters $\um, \vm>0$, 
suppose there are integers ${m,n\geq 0}$ such that
\begin{equation*}
\frac{m}{\um} + \frac{n}{\vm} = 1.
\end{equation*}
Then the nonnegative commutator relation $[f_{1/\um}, f_{\vm/\um}] \geq 0$ is satisfied.
\end{thm}

This  proposition will follow using  the following lemma, which is the special case $(m,n)=(1,1)$.
\begin{lem}[Rectangular Hyberbola Sufficiency]
\label{lem:hyperbola-pos}
If $(\um,\vm)$ lies on the rectangular hyperbola $\frac{1}{x} + \frac{1}{y} = 1$, then the lattice 
$\Lambda_{\um,\vm}$ is disjoint from the set $\mathcal{D}$.
\end{lem}

\begin{proof}
To show the lattice $\Lambda_{\um,\vm} = \{ (m\um, n\vm) : m,n \in\ZZ \}$ is disjoint from $\cD$,
it suffices to show they are disjoint in the positive quadrant since both sets are preserved by
  $(x,y) \mapsto (-x,-y)$ and $\cD$ does not intersect the second or fourth quadrants.
For any integers $m,n\geq 1$, the point $(m\um,n\vm)$ lies on the positive branch of the
 hyperbola $\frac{m}{x} + \frac{n}{y} = 1$. 
This curve is monotonically decreasing and intersects the diagonal at the integer point $(m+n, m+n)$.  
This implies the  hyperbola is disjoint from $\cD$, so
  $(m\um, n\vm)$ is not in $\cD$.
\end{proof}

\begin{proof}[Proof of  Theorem \ref{thm:pos-sufficient}]
There are three cases.
In each case, it suffices to show that condition (P2) of Proposition \ref{lem:diagonal-pos} 
holds as this is equivalent to the desired nonnegative commutator condition (P1).

{\em Case (i-c).}  Suppose $\um,\vm$ satisfy $\frac{m}{\um} + \frac{n}{\vm} = 1$ with both $m,n\geq 1$.
Setting $\um_0 = \frac{1}{m}\um$ and $\vm_0=\frac{1}{m}\vm$, 
we see that the lattice $\Lambda_{\um,\vm}$ is contained in the lattice $\Lambda_{\um_0, \vm_0}.$
By Lemma \ref{lem:hyperbola-pos} the lattice $\Lambda_{\um_0,\vm_0}$ does not intersect $\cD$, so neither does $\Lambda_{\um,\vm}$.

{\em Case(i-b).} If $n = 0$, then $\frac{m}{\um} = 1$ implies $\um= m$ is an integer. 
The region $\cD$ does not contain any points with integer $x$-coordinate, 
so $\Lambda_{\um,\vm}$ is disjoint from $\cD$. 

{\em Case(i-a).}  $m=0$ then $\frac{n}{\vm} = 1$ implies $\um = n$ is an integer, and the result follows by interchange of $\um$ and $\vm$ from the preceding case, using Lemma \ref{lem:symm-prop} (3). 
\end{proof}

Converting Theorem \ref{thm:pos-sufficient} from $(\um,\vm)$ coordinates  to $(\alpha,\beta)$ coordinates
 proves one direction of Theorem \ref{thm:positive}, 
 classifying positive solutions $\alpha,\beta$  satisfying 
 the nonnegative commutator relation $[f_\alpha,f_\beta]\geq 0$.

%
%
\section{Positive Dilations Classification: Necessity}
\label{sec:positive-N}
\setcounter{equation}{0}
In this section we prove the necessity of the criterion  of Theorem \ref{thm:positive}
for membership in $S$.
That is,  if positive dilations $\alpha, \beta$ satisfy $[f_\alpha, f_\beta] \geq 0$
then they necessarily  satisfy
\begin{equation}
\tag{\ref{eq:positive2}}
 m\alpha + n\frac{\alpha}{\beta} = 1 \quad\text{for some integers} \, m,n\geq 0.
\end{equation}


For this direction  of the proof we birationally transform the parameter space to 
a second set of coordinates, $(\talpha, \trho)$-coordinates, given by
\begin{equation}
(\talpha, \trho) := (\frac{1}{\um}, \frac{1}{\vm})= (\alpha, \frac{\alpha}{\beta}).
\end{equation}
Note that although $\talpha = \alpha$, 
we use a different variable name to remind ourselves that we are using a different coordinate system.
This map takes the open first quadrant of $(\alpha, \beta)$-coordinates onto the open first quadrant
in $(\talpha, \trho)$-coordinates, and its inverse map is  
$(\talpha, \trho) \to  (\talpha, \frac{\talpha}{\trho})= (\alpha, \beta)$.

Figure \ref{fig:rs-coord} pictures the positive dilation solutions of $S$ 
in  the $(\talpha, \trho)$-coordinate system.


\begin{figure}[h]
\begin{center}
\begin{tikzpicture}[scale=.50]
  \draw[->] (0,0) -- (10,0) node[right] {$\talpha=\alpha$};
  \draw[->] (0,0) -- (0,10) node[above] {$\trho = \alpha/\beta$};
  \draw[scale=8.6] (1,0.05) -- (1,-0.05) node[below] {$1$};
  \draw[scale=8.6] (0.05,1) -- (-0.05,1) node[left] {$1$};
  
  \draw[scale=8.6,domain=0:1.1,smooth,variable=\x, blue] plot ({\x},{1});
  \draw[scale=8.6,domain=0:1.1,smooth,variable=\x, blue] plot ({\x},{1/2});
  \draw[scale=8.6,domain=0:1.1,smooth,variable=\x, blue] plot ({\x},{1/3});
  \draw[scale=8.6,domain=0:1.1,smooth,variable=\x, blue] plot ({\x},{1/4});
  \draw[scale=8.6,domain=0:1.1,smooth,variable=\x, blue] plot ({\x},{1/5});
  
  \draw[scale=8.6,domain=0:1.1,smooth,variable=\y, red]  plot ({1},{\y});
  \draw[scale=8.6,domain=0:1.1,smooth,variable=\y, red]  plot ({1/2},{\y});
  \draw[scale=8.6,domain=0:1.1,smooth,variable=\y, red]  plot ({1/3},{\y});
  \draw[scale=8.6,domain=0:1.1,smooth,variable=\y, red]  plot ({1/4},{\y});
  \draw[scale=8.6,domain=0:1.1,smooth,variable=\y, red]  plot ({1/5},{\y});
  
  \draw[scale=8.6,domain=0:1,smooth,variable=\x, green] plot ({\x},{1-\x});
  \draw[scale=8.6,domain=0:0.5,smooth,variable=\x, green] plot ({\x},{1-2*\x});
  \draw[scale=8.6,domain=0:1,smooth,variable=\x, green] plot ({\x},{1/2-\x/2});
  \draw[scale=8.6,domain=0:0.33,smooth,variable=\x, green] plot ({\x},{1-3*\x});
  \draw[scale=8.6,domain=0:1,smooth,variable=\x, green] plot ({\x},{1/3-\x/3});
  \draw[scale=8.6,domain=0:0.5,smooth,variable=\x, green] plot ({\x},{1/2-\x});
\end{tikzpicture}
\end{center}
\caption{Positive dilation solutions of $S$  in $(\talpha,\trho)$-coordinates: 
$\talpha=\alpha, \, \trho= \frac{\alpha}{\beta}$.}
\label{fig:rs-coord}
\end{figure} 

In  $(\talpha, \trho)$ coordinates the set $S$ consists of straight lines: 
either vertical or horizontal half-infinite lines,
or oblique line segments connecting the two axes.

%
%
\subsection{Torus subgroup criterion: positive dilations}

To prove  sufficiency of the classification given in Theorem \ref{thm:positive}, 
 we establish a criterion for nonnegative commutator, given in
terms of  $(\talpha, \trho)$-coordinates.
 It is expressed in terms of 
a cyclic subgroup of the torus $\TT = \RR^2/\ZZ^2$ avoiding a set $\cCbar_{\talpha, \trho}$
which also varies with the parameters. 
The set to be avoided, viewed in $\RR^2$  
is an open rectangular region  with one corner at the origin. 

\begin{defi}\label{defn:corner-rectangle}
We define the {\em corner rectangle} 
$\cC_{\talpha,\trho}$ to be the  open  region
\begin{equation}
\cC_{\talpha,\trho} := \{(x,y) : 0<x<\talpha, \, 0<y<\trho \} \subset \RR^2,
\end{equation}
and let $\cC = \cC_{1,1}$, i.e.
\begin{equation}
\cC := {\{ (x,y) : 0 < x,y < 1\}} \subset \RR^2.
\end{equation}
\end{defi}

Given a real number $x$, we let $\widetilde{x}$ denote its image
 under the quotient map $\RR \to \RR/\ZZ$.
We project $\cC_{\talpha,\trho}$ to  
the torus $\TT= \RR^2/\ZZ^2$,
and let $\cCbar_{\talpha,\trho}$ denote  the image of $\cC_{\talpha,\trho}$
via coordinatewise projection  $\RR^2 \to \TT$.
 The region  $\cCbar_{\talpha,\trho}$ in the $\TT$ is either an open rectangle 
(if $\talpha,\trho\leq 1$), or an open annulus (if either $\talpha>1$ or $\trho>1$), or the whole torus (if both $\talpha,\trho>1$).
See Figure \ref{fig:torus-corner-rectangle}.


\begin{figure}[h]
\begin{center}
  \begin{tikzpicture}[scale=3]
  \draw[-] (4/9,0) -- (1,0);
  \draw[-] (0,1/3) -- (0,1);
  \draw[dotted] (1,0) -- (1,1);
  \draw[dotted] (0,1) -- (1,1);
  
  \node[draw,circle,inner sep=0.5pt,fill,red] at (4/9,1/3) {};
  \draw[-] (4/9,1/3) node[right] {$(\talpha, \trho)$};
  
  \draw[dashed] (0,0) -- (4/9,0);
  \draw[dashed] (0,1/3) -- (4/9,1/3);
  \draw[dashed] (0,0) -- (0,1/3);
  \draw[dashed] (4/9,0) -- (4/9,1/3);
    
  \fill[gray,nearly transparent] (0,0) -- (0,1/3) -- (4/9,1/3) -- (4/9,0) -- cycle; 
  \end{tikzpicture}
  \qquad\qquad
  \begin{tikzpicture}[scale=3]
  \draw[-] (0,0) -- (1,0);
  \draw[dotted] (1,0) -- (1,1);
  \draw[dotted] (0,1) -- (1,1);
  
  \node[draw,circle,inner sep=0.5pt,fill,red] at (2/5,6/5) {};
  \draw[-] (2/5,6/5) node[right] {$(\talpha, \trho)$};
  
  \draw[dashed] (0,0) -- (0,1);
  \draw[dashed] (2/5,0) -- (2/5,1);
  
  \fill[gray,nearly transparent] (0,0) -- (0,6/5) -- (2/5,6/5) -- (2/5,0) -- cycle; 
  \end{tikzpicture}
  \qquad\qquad
  \begin{tikzpicture}[scale=3]
  \draw[-] (0,0) -- (1,0);
  \draw[-] (0,0) -- (0,1);
  \draw[dotted] (1,0) -- (1,1);
  \draw[dotted] (0,1) -- (1,1);
  
  \node[draw,circle,inner sep=0.5pt,fill,red] at (8/7,6/5) {};
  \draw[-] (8/7,6/5) node[right] {$(\talpha, \trho)$};
  
  \fill[gray,nearly transparent] (0,0) -- (0,6/5) -- (8/7,6/5) -- (8/7,0) -- cycle; 
  \end{tikzpicture}
\end{center}
\caption{Corner rectangles $\cCbar_{\talpha, \trho}$ in the torus $\TT= \RR^2/\ZZ^2$ for  varying parameters 
$(\talpha, \trho)$.}
\label{fig:torus-corner-rectangle}
\end{figure} 

\begin{prop}[Nonnegative commutator relation: Torus subgroup criterion]
\label{prop:torus-pos}
For $\talpha, \trho>0$, the following conditions are equivalent.
\begin{enumerate}
\item[(Q1)] The nonnegative commutator relation holds:
\[
 [f_{\talpha}, f_{\talpha/\trho}] \geq 0 .
\] 

\item[(Q2)] The cyclic torus subgroup 
\[
\langle (\talpha, \trho) \rangle_{\TT} = {\{ (\widetilde{n\talpha}, \widetilde{n\trho})  : n\in \ZZ \}}  \subset \RR^2 /\ZZ^2 = \TT
\]
is disjoint from the  corner rectangle 
$\cCbar_{\talpha,\trho}= \{(\widetilde{x},\widetilde{y}) : 0<x<\talpha, \, 0<y<\trho \}$
in  $\TT$.
\end{enumerate}
\end{prop}

\begin{proof}
Let $\um = \talpha^{-1}$ and $ \vm = \trho^{-1}$.
By Proposition \ref{lem:diagonal-pos} the commutator relation $[f_{\talpha}, f_{\talpha/\trho}] \geq 0$ 
is equivalent to the lattice $\Lambda_{\talpha^{-1},\trho^{-1}}$ 
being disjoint from $\cD$ as subsets of $\RR^2$.
The lattice $\Lambda_{\talpha^{-1},\trho^{-1}}$ is the image of the map 
\begin{align*}
\varphi_{\talpha,\trho} : \ZZ^2 &\to \RR^2 \\
(m,n) &\mapsto (m\talpha^{-1}, n\trho^{-1} ) .
\end{align*}
Note that any point in $\cD$ 
may be uniquely expressed as the sum of a point $(k,k)$ for some integer $k$
and a point inside the open unit square 
$ \cC = {\{ (x,y)\in \RR^2 : 0 < x,y < 1\}} $.
Therefore the image of $\varphi_{\talpha,\trho}$ is disjoint from $\cD$ if and only if the image of 
$\varphi_{\talpha,\trho,1}$ is disjoint from $\cC$,
where we let
\begin{align*}
\varphi_{\talpha,\trho,1} : \ZZ^3 &\to \RR^2  \\
(m,n,k) &\mapsto (m\talpha^{-1} + k, n\trho^{-1} + k) .
\end{align*}
(Recall, by Proposition \ref{lem:diagonal-pos}, that this happens if and only if 
$[f_{1/\um}, f_{\vm/\um}] = [f_\talpha, f_{\talpha/\trho}] \geq 0$.)

Making a rescaling of $\RR^2$ by $(x,y) \mapsto (\talpha x, \trho y)$, 
the condition $[f_\talpha, f_{\talpha/\trho}] \geq 0$
is equivalent to the disjointness of the image of
\begin{align*}
\phi_{\talpha,\trho,1} : \ZZ^3 &\to \RR^2 \\
(m,n,k) &\mapsto (m + k\talpha, n + k\trho) 
\end{align*}
from the open corner rectangle
$$
\cC_{\talpha,\trho} = \{(x,y)\in \RR^2 :  0 < x < \talpha,\, 0< y < \trho \}.
$$
The image of $\phi_{\talpha,\trho,1}$  inside $\RR^2$ 
 is  disjoint from $\cC_{\talpha,\trho}$
 if and only if their projections under $\RR^2 \to \RR^2 / \ZZ^2$ remain disjoint, 
 since the image $\phi_{\talpha,\trho,1}(\ZZ^3)\subset\RR^2$ 
 is clearly invariant under $\ZZ^2$-translations.
The projection of $\phi_{\talpha,\trho,1}(\ZZ^3)$ to the torus $\TT = \RR^2 / \ZZ^2$ 
is the cyclic torus subgroup $\langle (\widetilde{\talpha},\widetilde{\trho}) \rangle_{\TT}$ 
generated by the image of $(\talpha,\trho)$.
The projection of $\cC_{\talpha, \trho}$ to the torus is $\cCbar_{\talpha,\trho}$.
This proves the equivalence of the disjointness criterion (Q2) with the nonnegative commutator relation (Q1).
\end{proof}

%
%
\subsection{Proof of necessity in Theorem \ref{thm:positive}}
\label{subsec:pos-necessary}

The necessity condition in Theorem \ref{thm:positive} states that if
positive $\alpha, \beta$ have $(\alpha, \beta) \in S$, then necessarily
$ m \alpha + n \frac{\alpha}{\beta} =1$ for some  integers $m, n \ge 0$.
Restated in the  $(\talpha,\trho)$ coordinates, this is equivalent to the following proposition.

\begin{thm}[Necessity  in $(\talpha, \trho)$-coordinates]
\label{thm:pos-necessary}
Suppose parameters $\talpha, \trho>0$ satisfy
the nonnegative commutator relation $[f_{\talpha}, f_{\talpha/\trho}] \geq 0$.
Then there are integers $m,n\geq 0$ such that
\begin{equation*}
{m}\talpha + {n}\trho = 1.
\end{equation*}
\end{thm}

 We prove Theorem \ref{thm:pos-necessary}
in the remainder of Section \ref{subsec:pos-necessary}. 
The analysis concerns cyclic subgroups of the torus, using the criterion in 
Proposition  \ref{prop:torus-pos}.
In the next subsection we recall two useful facts for the analysis.

%
%
\subsubsection{Technical Tools}\label{sec:621}
We state two technical results used in the following proofs.

\begin{thm}[Closed subgroup theorem]
\label{thm:lie-subgroup}
Given a Lie group $G$ and a subgroup $H\subset G$, 
the topological closure $\bar{H}$ of the subspace $H$ is a Lie subgroup $\bar{H}\subset G$.
\end{thm}
\begin{proof}
See  Lee \cite[Theorem 20.12, p. 523]{Lee13}.
Under the given hypotheses, $\bar{H}$ is a closed subgroup of $G$.
\end{proof}

\begin{thm}[Sylvester duality theorem]
\label{thm:sylvester}
Given $a,b$  coprime positive integers, let 
$$\sS := \sS(a, b) = a\NN + b\NN$$ 
denote the semigroup  generated by $a$ and $b$
inside $\NN = \{0, 1, 2, \ldots\}$.
Then $n\in \sS$ if and only if $ab - a - b - n\not\in \sS$.
\end{thm}
\begin{proof}
See  Beck and Robins \cite[Theorem 1.3, p. 6]{BeckR07} and the discussion thereafter, on pp. 12-14.
\end{proof}

\noindent The name Sylvester is associated to this result from his  problem
 \cite{Sylvester:1884}   posed in 1884.

\begin{rmk}\label{rem:59}
Given positive integers $a_1, a_2, \ldots a_n$ having $\gcd(a_1, a_2, ..., a_n)=1$,
let $\sS= \sS(a_1, a_2, ..., a_n)$ denote the (finitely generated) nonnegative integer semigroup 
generated by  $a_1, \ldots, a_n$.
The $\gcd$ condition implies that $\sS$ includes all sufficiently large integers.
We  call  
$$NR(a_1,\ldots,a_n) :=  \NN \smallsetminus \sS(a_1,\ldots,a_n) $$ 
 the {\em Frobenius non-realizing set} of the  generators $a_1,\ldots,a_n$,
 and its largest member  is called the {\em Frobenius number} of $\sS$. 
 The {\em Diophantine Frobenius problem} is the
 problem of determining the Frobenius number as a function of $(a_1, a_2, ..., a_n)$.
This problem  has been extensively studied,  see Ram\'{i}rez Alfons\'{i}n \cite{RA05}.
In  the two-generator case $\sS= \sS(a, b)$ the Frobenius number is $ab-a-b$,
a result implied by  Sylvester duality. 
 In 2006 Tuenter \cite{Tue06} gave a characterization of the complete set $NR(a, b)$ in
 terms of  its moments.
\end{rmk}
%
\subsubsection{Proof of necessity direction}

\begin{proof}[Proof of Theorem \ref{thm:pos-necessary}]
We divide the proof into two cases, depending on whether or not both the parameters 
$\talpha, \trho$ are  rational.
In each case, 
we use the torus disjointness condition (Q2) of Proposition \ref{prop:torus-pos} 
in place of the nonnegative commutator condition (Q1).
Namely, it suffices to show that for (Q2) to hold,
it is necessary that
$m\talpha + n\trho = 1$ for some integers $m,n\geq 0$.

{\em Case 1:}
{\em At least one of $\talpha, \trho$ is irrational.}
In this case  the map
\begin{align*}
\widetilde{\phi}_{\talpha,\trho} : \ZZ &\to \TT \\
k &\mapsto (\widetilde{k\talpha}, \widetilde{k\trho})
\end{align*}
is injective and 
$H = \langle (\widetilde{\talpha},\widetilde{\trho}) \rangle_{\TT} = \im(\widetilde{\phi}_{\talpha,\trho})$ 
is an infinite cyclic subgroup of the torus $\TT$. 
Since $\TT$ is compact, $H$ cannot be a discrete subset of points.
Thus by Theorem \ref{thm:lie-subgroup} the closure of this subspace is a 
Lie subgroup $\bar{H}$ of dimension 1 or 2. 
If $\bar{H}$ is dimension 2 then $H$ is dense in $\TT$ 
and will intersect the non-empty open rectangle $\cCbar_{\talpha,\trho}$, so condition (Q2) fails.

If $\bar{H}$ is dimension 1, then 
the subgroup $\bar{H}$ is the projection to $\TT$ of the lines 
$\{ (x,y)\in \RR^2 : mx + ny \in \ZZ \}$
for some integers $m,n$.
(The subgroup $\bar{H}$ determines the pair $(m,n)$ almost uniquely, up to a sign $(-m,-n)$;
the ratio $-\frac{m}{n}$ is the slope of the lines in $\bar{H}$, 
and $\gcd(m,n)$ is the number of connected components of $\bar{H}$.)
The parameters $\talpha,\trho$ must satisfy some integer relation
\begin{equation*}
m\talpha + n\trho = k, \quad m,n,k \in \ZZ .
\end{equation*}
We must have $\gcd(m,n,k) = 1$
since the cyclic subgroup $H = \langle (\widetilde{\talpha},\widetilde{\trho}) \rangle_{\TT}$ 
is assumed to be dense in $\bar{H} = \{(\widetilde{x},\widetilde{y})\in \TT : mx + ny \in \ZZ\}$; 
otherwise, $\gcd(m,n,k)$ would give the index of $\bar{H}$ inside 
$\{(\widetilde{x},\widetilde{y})\in \TT : mx + ny \in \ZZ\}$.

If the integers $m,n$ have opposite sign, 
then the lines in $\bar{H}$ will have positive slope and $H$ will intersect 
$\cCbar_{\talpha,\trho}$ in a neighborhood of $(0,0)$. 
See Figure \ref{fig:torus-pos-slope}.
\begin{figure}[h]
\begin{center}
\begin{tikzpicture}[scale=3]
  \draw[-] (0,0) -- (1,0);
  \draw[-] (0,0) -- (0,1);
  \draw[dotted] (1,0) -- (1,1);
  \draw[dotted] (0,1) -- (1,1);
  
  \draw[domain=0:1,smooth,variable=\x, blue] plot ({\x},{2/3*\x});
  \draw[domain=0:1/2,smooth,variable=\x, blue] plot ({\x},{2/3*\x+2/3});
  \draw[domain=1/2:1,smooth,variable=\x, blue] plot ({\x},{2/3*\x-1/3});
  \draw[domain=0:1,smooth,variable=\x, blue] plot ({\x},{2/3*\x+1/3});
  
  \node[draw,circle,inner sep=0.5pt,fill,red] at (2/5,4/15) {};
  \draw[-] (2/5,5/15) node[above] {$(\talpha,\trho)$};
  
  \fill[gray,nearly transparent] (0,0) -- (0,4/15) -- (2/5,4/15) -- (2/5,0) -- cycle; 
  \end{tikzpicture}
\end{center}
\caption{Solutions to $2x-3y\in\ZZ$ in the torus $\TT = \RR^2/\ZZ^2$}
\label{fig:torus-pos-slope}
\end{figure} 

Otherwise $m,n$ have the same sign (including the case where $m$ or $n$ is zero) and we may assume that $m,n,k$ are all nonnegative. 
If $k\geq 2$, then the point $(\frac1{k}\talpha, \frac1{k}\trho)$ lies in the closed subgroup $\bar{H}$ and in the open rectangle $\cCbar_{\talpha, \trho}$. 
See Figure \ref{fig:torus-neg-slope}.
This implies $H$ must also intersect the open region $\cCbar_{\talpha,\trho}$,
so condition (Q2) fails.
Thus $k = 1$ is necessary for (Q2) to hold in this case.

\begin{figure}[h]
\begin{center}
  \begin{tikzpicture}[scale=3]
  \draw[-] (0,0) -- (1,0);
  \draw[-] (0,0) -- (0,1);
  \draw[dotted] (1,0) -- (1,1);
  \draw[dotted] (0,1) -- (1,1);
  \draw[domain=0:2/3,smooth,variable=\x, blue] plot ({\x},{1-3/2*\x});
  \draw[domain=2/3:1,smooth,variable=\x, blue] plot ({\x},{2-3/2*\x});
  \draw[domain=0:1/3,smooth,variable=\x, blue] plot ({\x},{1/2-3/2*\x});
  \draw[domain=1/3:1,smooth,variable=\x, blue] plot ({\x},{3/2-3/2*\x});
  
  \node[draw,circle,inner sep=0.5pt,fill,red] at (4/9,1/3) {};
  \draw[-] (4/9,1/3) node[right] {$(\talpha, \trho)$};
  
  \node[draw,circle,inner sep=0.5pt,fill,red] at (2/9,1/6) {};
  \draw[-] (2/11,1/6) node[left] {$(\frac{1}{2} \talpha, \frac{1}{2} \trho)$};
  
  \fill[gray,nearly transparent] (0,0) -- (0,1/3) -- (4/9,1/3) -- (4/9,0) -- cycle; 
  \end{tikzpicture}
\qquad\qquad
  \begin{tikzpicture}[scale=3]
  \draw[-] (0,0) -- (1,0);
  \draw[-] (0,0) -- (0,1);
  \draw[dotted] (1,0) -- (1,1);
  \draw[dotted] (0,1) -- (1,1);
  \draw[domain=0:2/3,smooth,variable=\x, blue] plot ({\x},{1-3/2*\x});
  \draw[domain=2/3:1,smooth,variable=\x, blue] plot ({\x},{2-3/2*\x});
  \draw[domain=0:1/3,smooth,variable=\x, blue] plot ({\x},{1/2-3/2*\x});
  \draw[domain=1/3:1,smooth,variable=\x, blue] plot ({\x},{3/2-3/2*\x});
  
  \node[draw,circle,inner sep=0.5pt,fill,red] at (2/9,1/6) {};
  \draw[-] (2/9,1/6) node[right] {$(\talpha, \trho)$};
  
  \fill[gray,nearly transparent] (0,0) -- (0,1/6) -- (2/9,1/6) -- (2/9,0) -- cycle; 
  \end{tikzpicture}
\end{center}
\caption{Solutions to $3x+2y\in\ZZ$ in the torus $\TT$; with $k=2$ and $k=1$}
\label{fig:torus-neg-slope}
\end{figure} 

{\em Case 2:}
{\em The parameters $\talpha,\trho$ are both rational.}
We may express $(\talpha,\trho) = (s\frac{a}{b}, t\frac{a}{b})$ 
where $s,t$ are coprime positive integers and $a,b$ are coprime positive integers.
In this case the subgroup
$$H = \langle (\widetilde{\talpha},\widetilde{\trho})\rangle_{\TT} 
 = \{ (\widetilde{\frac{nsa}{b}}, \widetilde{\frac{nta}{b} }) \in \TT : n\in \ZZ \}$$
is a discrete subspace of the torus, containing precisely $b$ points. 
By the assumption that $\gcd(a, b) =1$, we have
$$H = \langle (\widetilde{\frac{s}{b}}, \widetilde{\frac{t}{b}})\rangle_{\TT} 
= \{ \left(\widetilde{ \frac{ns}{b}}, \widetilde{ \frac{nt}{b}}\right): 0 \le n \le b-1\}.$$ 

{\em Case 2-a:} $a \geq 2$. 

In this case   $ (\widetilde{\frac{s}{b}}, \widetilde{\frac{t}{b}})$ belongs to  $H$ 
and lies inside the open  corner rectangle $\cCbar_{\talpha,\trho}$.
 Thus condition (Q2) does not hold.

{\em Case 2-b:}   $a = 1$.

In this case,  we have  $(\talpha,\trho) = (\frac{s}{b}, \frac{t}{b})$.
The proposition asserts
$(\talpha,\trho) = (\frac{s}{b}, \frac{t}{b})$ 
satisfies the commutator inequality relation
$[f_\talpha, f_{\talpha/\trho}]\geq 0$ 
only if there exist nonnegative  integers $m,n$ such that
\[ m\frac{s}{b} + n\frac{t}{b} = 1,\]
or equivalently, only if $b \in s\NN + t\NN$, 
the semigroup generated by $s$ and $t$ inside $\NN = \{0, 1, 2, \ldots\}$.
To prove this  we consider its contrapositive. 

To show the contrapositive, suppose  $b \not\in \NN s + \NN t$. 
Then by the Sylvester Duality Theorem \ref{thm:sylvester} we have  
\[ st - s - t - b \in s\NN  + t\NN. \]
In consequence $1 \le b < st$ and we can express 
\[ st - b = ms + nt \]
for some  {positive} integers $m,n \geq 1$.
The integers necessarily satisfy $1 \le m < t$ and $1 \le n < s$, since $ 1\leq st-b < st.$

Since $\gcd(s, t)=1$ there exist  (possibly negative) integers $m_0$, $n_0$  satisfying
\[ 
1 = m_0 s + n_0 t .
\]
Note that  $m_0$ is a multiplicative inverse of $s$ modulo $t$, 
and similarly  $n_0$ is inverse to $t$ modulo $s$.
We now assert the following claim:
 
{\bf Claim.}  {\em The point  $(\widetilde{\frac{s-n}{b}}, \widetilde{\frac{m}{b}})$ is in the torus subgroup 
$H = \langle (\widetilde{\frac{s}{b}}, \widetilde{\frac{t}{b}})\rangle_{\TT}$.
Specifically, we  show that in the torus $\TT$, 
the point $(\widetilde{\frac{s-n}{b}}, \widetilde{\frac{m}{b}})$ is equal to the multiple 
$(\widetilde{N\frac{s}{b}},\widetilde{N\frac{t}{b}})$
where 
$$N = m_0(s-n) + n_0 m .$$}

First we verify that $N$ satisfies the following relations to $b$: since 
$$b = st - ms - nt = (s-n)t - ms  $$
we have 
\begin{align*} 
m_0 b &=  m_0(s-n)t  -m m_0 s \\
&=  m_0(s-n)t + m(n_0 t-1)  \\
&= ( m_0 (s - n) + n_0 m )t -m = Nt - m
\end{align*}
and
\begin{align*}
-n_0 b  &= - (s-n)n_0 t + n_0 m s \\
&= (s-n)(m_0 s -1) + n_0 m s  \\
&= ( m_0(s-n) + n_0m) s - (s-n) = Ns - (s-n).
\end{align*}
This shows that
$$ N = \frac{1}{t}(m_0b +m) = \frac{1}{s} (-n_0b + s-n) .$$

Second,  consider the multiple 
$(N\talpha,N\trho) = (N\frac{s}{b}, N\frac{t}{b})$. 
We have
\begin{align*}
N\frac{s}{b} &= \left(\frac{-n_0b+s-n}{s}\right) \frac{s}{b} = -n_0 + \frac{s-n}{b}
\end{align*}
so $N\frac{s}{b} \equiv \frac{s-n}{b}$ modulo 1,
and
\begin{align*}
N\frac{t}{b} = \left(\frac{m_0 b+m}{t}\right) \frac{t}{b}
= m_0 + \frac{m}{b}
\end{align*}
so $N\frac{t}{b} \equiv \frac{m}{b}$ modulo 1.
This shows that $(\widetilde{\frac{s-n}{b}}, \widetilde{\frac{m}{b}}) 
= (\widetilde{N\frac{s}{b}},\widetilde{N\frac{t}{b}})$ 
is in the torus subgroup 
$H = \langle(\widetilde{\talpha}, \widetilde{\trho}) \rangle_{\TT}$,
proving the claim. 

Now the claim certifies that 
$(\talpha, \trho) = (\frac{s}{b}, \frac{t}{b})$ fails to satisfy property
(Q2) in Proposition \ref{prop:torus-pos}, since 
$(\widetilde{\frac{s-n}{b}}, \widetilde{\frac{m}{b}}) \in H$,
while the bounds
$$ 0 < s-n < s \quad\text{and}\quad 0<m<t $$
show that  $(\widetilde{\frac{s-n}{b}}, \widetilde{\frac{m}{b}})$  
lies in the open rectangle $\cCbar_{\talpha,\trho}$.
We conclude  property  (Q1) does not hold.
The contrapositive is proved, completing Case 2-b.
\end{proof}

By converting Theorem \ref{thm:pos-necessary} from $(\talpha,\trho)$ coordinates
to $(\alpha,\beta)$ coordinates, we have the proof of the other direction of Theorem \ref{thm:positive}.
Thus Theorem \ref{thm:positive} is now proved,
by combining the statements of Theorem \ref{thm:pos-sufficient} and Theorem \ref{thm:pos-necessary}.

%
%
\section{Proof of Closure Theorem \ref{thm:closed} for Positive and Mixed Sign Dilations}
\label{sec:consequences}
\setcounter{equation}{0}


In this section we prove the closure property of Theorem \ref{thm:closed} 
restricted to the closed set  of nonnegative dilations and mixed sign dilations, on 
the closed region $\RR^2_{nn}$ where at least one coordinate is nonnegative.
\begin{thm}
\label{thm:closed2}
Let $S$ be the set of all dilation factors $(\alpha, \beta)$ 
which satisfy the nonnegative commutator inequality ${[f_\alpha, f_\beta]\geq 0 }$,
where $f_\alpha(x) = \floor{\alpha x}$. Let 
$$\RR^2_{nn}= \{ (\alpha, \beta) \in \RR^2 : \alpha \ge 0 \mbox{ or } \beta \ge 0\}.$$
Then $S \cap \RR^2_{nn}$ is a closed subset of $\RR^2_{nn}$. 
\end{thm}

\begin{proof}
To show that $S \subset \RR^2_{nn}$ is closed, 
is suffices to check that its intersection with each closed quadrant is closed. 
The set $\RR^2_{nn}$ is the union of three  closed quadrants. 
Since the two coordinate axes are entirely  in $S$,  
it suffices to check that $S$ is closed in each open quadrant.
For the second and fourth quadrants, this is immediate by Theorem \ref{thm:mixed-sign}.

In the open first quadrant $\RR^2_{+}$, let $S_{I} = \{(\alpha,\beta) > 0 :  [f_\alpha,f_\beta]\geq 0\}$.
We use Proposition \ref{lem:diagonal-pos} in $(\um, \vm)$-coordinates to prove that $S_{I}$ is closed.
The birational map 
$(\um,\vm) \mapsto (\frac{1}{\um},\frac{\vm}{\um}) = (\alpha, \beta)$ 
is a homeomorphism from the open first quadrant to itself, 
so it suffices to show that 
$$ {S}_{I}^{\um,\vm} = \{ (\um,\vm) > 0:   [f_{1/\um},f_{\vm/\um}] \geq 0\} $$
is closed inside the open first quadrant $\RR^2_{+}$.
Now the lattice disjointness criterion in Proposition \ref{lem:diagonal-pos}
yields  
\begin{align*}
{S}_{I}^{\um,\vm} &= \{ (\um,\vm)>0 : (m \um,n\vm) \not\in \cD \text{ for all }m,n\in \ZZ\}\\
 &= \{ (\um,\vm)>0 : (\um,\vm)\not\in\cD^{(m,n)} \text{ for all }m,n\in \ZZ \},\\
 &= \RR^2_{+} \smallsetminus \bigcup_{m,n} \cD^{(m,n)}.
\end{align*} 
Here  $\cD^{(m,n)} := \{(x,y) : (mx, ny)\in \cD \}$, where 
the region $\mathcal{D}:= \{(x, y): \floor{y} < x < \ceil{y} \}$ is an
open set in $\RR^2$. 
Thus each  $\cD^{(m,n)}$   is an open set in $\RR^2$,
since $\cD^{(m,n)}$  is homeomorphic to $\cD$ if $m,n\neq 0$, 
and it is empty otherwise.
It follows that
${S}_{I}^{\um,\vm}$ is the complement of an open set in $\RR^2_{+}$ 
so is closed, as desired.
\end{proof}

\section*{Acknowledgements}
We are indebted to David Speyer for his observation that our results implied the
preorder property  stated in  Theorem \ref{thm:poset}.
The second author thanks  Yilin Yang for helpful conversations.


\end{document}